\documentclass[11pt,reqno]{amsart}
\usepackage{graphicx}
\usepackage{amsmath,amssymb}
\usepackage{color, tikz}
\usepackage[utf8]{inputenc}

\newtheorem{thm}{Theorem}[section]

\newtheorem{lem}[thm]{Lemma}

\newcommand{\up}{c_{\parallel}}
\newcommand{\un}{c_{\notparallel}}
\theoremstyle{definition}

\theoremstyle{remark}
\newtheorem{rem}[thm]{Remark}

\numberwithin{equation}{section}

\renewcommand{\parallel}{\mathrel{/\mkern-5mu/}}
\makeatletter
\newcommand{\notparallel}{%
  \mathrel{\mathpalette\not@parallel\relax}%
}
\newcommand{\not@parallel}[2]{%
  \ooalign{\reflectbox{$\m@th#1\smallsetminus$}\cr\hfil$\m@th#1\parallel$\cr}%
}
\makeatother

 \theoremstyle{plain}

\newcommand{\norm}[1]{\left\Vert#1\right\Vert}

\newcommand{\T}{{\mathbb T}}

\usepackage[mathcal]{euscript}
\usepackage{amstext}
\usepackage{mathrsfs}
\usepackage{bm}

\newcommand{\RR}{\mathbb{R}}
\newcommand{\NN}{\mathbb{N}}

\newcommand{\TT}{\mathbb{T}}

\newcommand{\bv}{\bm{v}}

\definecolor{Green}{rgb}{0.010,0.7,0.02}

\usepackage{enumitem}
\usepackage{latexsym}
\usepackage[sort,nospace,noadjust]{cite}
\usepackage[pagebackref=true, colorlinks=true, citecolor=blue, linkcolor=red]{hyperref}

\usepackage{mathtools}

\newcommand{\beq}{\begin{equation}}
\newcommand{\eeq}{\end{equation}}

\newcommand{\bal}{\begin{align}}
\newcommand{\eal}{\end{align}}

\newcommand{\rL}{\mathring{L}}

\begin{document}

\title[Phase separation for Cahn-Hilliard with a shear]{Phase separation for the 2D Cahn-Hilliard equation with a background shear flow}

\author[Y. Feng {\em et Al.}]{Yu Feng}%
\address{Yu Feng:  School of Sciences, Great Bay University, Dongguan, China, 523000, P.R.China }
\email{fengyu@gbu.edu.cn}

\author[]{Yuanyuan Feng}
\address{Yuanyuan Feng: School of Mathematical Sciences, Key Laboratory of MEA (Ministry of Education) and Shanghai Key Laboratory of PMMP, East China Normal University, Shanghai, 200241, P.R. China }
\email{yyfeng@math.ecnu.edu.cn}

\author[]{Anna L. Mazzucato$\,^\ast$}
\address{Anna L. Mazzucato: Department of Mathematics, Penn State University, University Park, PA, 16802, U.S.A.}
\email{alm24@psu.edu}

\author[]{Xiaoqian Xu}
\address{Xiaoqian Xu: Duke Kunshan University, Zu Chongzhi Center, No. 8 Duke Avenue, Kunshan, Jiangsu Province, 215316, P.R. China}
\email{xiaoqian.xu@dukekunshan.edu.cn}

\date{\today}

\subjclass[2020]{35K91, 76F25, 76E05}%

\keywords{Cahn-Hilliard, phase separation, shear flow, enhanced dissipation, mixing}%

\thanks{$\,^\ast$ Corresponding author.}

\begin{abstract}
We consider the Cahn-Hilliard equation, which models phase separation in binary fluids, on the two-dimen\-sional torus in the presence of advection by a given background shear flow, satisfying certain conditions and of sufficiently large amplitude. By exploiting the resulting enhanced dissipation for the linearized operator, we prove that, with well-prepared data, the solution converges asymptotically at large times to the solution of a one-dimensional Cahn-Hilliard equation, obtained by projecting the full equation in the direction orthogonal to the shear in a suitable sense. This result rigorously justified the observed phenomenon of striation in the concentration field. 
\end{abstract}


\maketitle


\mathtoolsset{showonlyrefs}


\section{Introduction and Main Result}
\subsection{The model and prior results}
This article concerns the effect of adding linear advection  by a shear flow to the Cahn-Hilliard equation on the two-dimensional torus $\TT^2\subset \RR^2$, specifically in regards to the phenomenon of phase separation at large times.


The Cahn-Hilliard equation (CHE for short) is a fourth-order semilinear parabolic equation that models phase separation in binary miscible fluids, especially binary alloys, and  exhibits pattern formation.
It takes the following form:
\begin{equation}
\label{eqn:standard_CH}
\partial_t c =\nu \Delta(c^3-c-\mu \Delta c).
\end{equation}
where $c=c(t,x)$ is a scalar function representing the concentration of each phase, $\nu>0$ is a physical parameter related to the mobility of the mixture, and $0<\sqrt{\mu}\ll 1$, the so-called  {\em Cahn number}, characterizes the thickness of the transition region between the two phases. The concentration is normalized so that regions where $c=\pm 1$ corresponds to separated phases. The mass of the system is then given by the total integral of the solution $c$, which is preserved under the CHE time evolution.

We denote a point $x$ in $\RR^2$ with $x=(x_1,x_2)$ and identify $\TT^2$ with the unit square $[0,1]^2$ in $\RR^2$, that is, we keep the period fixed. We then impose periodic conditions at  the boundary of the unit square. We also supplement \eqref{eqn:standard_CH} with an initial condition $c(0,x)=c_0(x)$.

As with other models for interface dynamics, notably the Kuramo\-to-Sivashinksy equation (KSE for short), the linearized operator $-\nu (\mu\Delta^2 +\Delta)$ in \eqref{eqn:standard_CH} is characterized by growing modes depending on the size of the periodic domain and lacks a maximum principle, which makes the analysis of the equation more challenging.  However, differently than the KSE,  the CHE has a gradient structure. This structure allows to establish global well-posedness of the CHE, while for the KSE  global existence in more than one space dimension is still essentially open. Indeed, the CHE dynamics can be understood as the gradient flow of a free energy $E$,  given by the following Ginzburg-Landau-type functional:
\begin{equation}
\label{eqn: energy functional}
    E[c]=\int_{\T^2}\frac{1}{4}(c^2-1)^2+\frac{\mu}{2}\vert\nabla c\vert^2 d x,
\end{equation}
which combines a double-well  potential energy with an interfacial energy depending on $\mu$.
In fact,  one  has that
\begin{equation*}
    \frac{d E}{dt}=-\nu\int_{\T^2}\left\vert\nabla\left(c^3-c-\mu\Delta c\right)\right\vert^2 dx<0.
\end{equation*}
The decay of the energy then gives global-in-time wellposedness under various boundary conditions for both weak and strong solutions (\cite{ES86,EF87}, see also \cite{Temam88}). A semigroup approach allows also to establish that the CHE admits a global, compact and connected attractor in Sobolev spaces (\cite{NST89}, see also \cite{CD94,CD96,Dlot93}). It is then natural to ask what is the long-time behavior of the system. Since the energy acts as a Lyapunov function, generically the system should settle onto a minimizer of the energy $E$. Formally, the minimizers correspond to equilibria with minimal interfacial thickness and maximum homogeneity. Therefore, one expects separation of phases into distinct regions and coarsening of such regions if $\mu$ is small enough, an observed phenomenon called {\em spinoidal decomposition} in alloys \cite{Cah61} (we refer also to \cite{KLCLJ16} for a survey of recent results).  
This dynamics has been investigated analytically, numerically and asymptotically \cite{ES86,EF87,Ell88,Peg89}. In the one-dimensional case, keeping the total mass fixed, if the initial concentration has both phases balanced, the system will settle in a relatively short time onto configurations of alternating homogeneous phases with a transition gap of optimal width. Such configurations are metastable and can persists for exponentially long times, but eventually nearby  intervals with the same phase merge and the final configuration consists in a number of alternating phases determined by the mass \cite{BS22,BH92,BFK18,OR07,ACRT05}.

When the CHE is used to model multiphase flow, it should be coupled to the Navier-Stokes equations, modeling viscous, incompressible flow. In a first approximation, the coupling with the fluid equations can be replaced by the action of a strong background flow, which adds a linear transport term to the CHE. We refer to this modified equation as the advective Cahn-Hilliard equation or ACHE.  When the advecting flow is mixing, the combined effect of stirring and (hyper)diffusion leads to enhanced dissipation and phase separation is suppressed provided the stirring is sufficiently strong \cite{FFIT19,NT07,OT07,OT08}. In the case of a background shear flow, the effect of advection of phase separation has been extensively studied numerically, primarily through molecular dynamics simulations, but also asymptotically, and experimentally for binary alloys \cite{OSN15,LDEBH13,Shear1,Shear2,Shear3,Shear4,Shear5,Shear6,Shear7,Shear8,Shear9,Shear10,Shear11} and in multiphase fluid flows \cite{Berth01,Bra03,SC00,HMMO95}. Even though the transient behavior can be complex, at sufficiently large times and when the strength of the shear is strong enough, the isotropic patterns typical of the two-dimensional CHE are distorted into elongated patterns along the direction of the shear and ultimately banded patterns appear, with a cross-sectional structure that resembles that for the one-dimensional CHE.

In this work, we study the long-time behavior of the ACHE under the action of an external steady shear flow, which for convenience  is assumed to be a horizontal shear. That is, the corresponding velocity field $\bv$ takes the form $\bv(x_1,x_2)=(v(x_2),0)$. We refer to $v$ as the {\em shear profile}, which we assume to be sufficiently regular. We discuss the regularity of the flow in more details later in the paper. With this choice, since the streamlines of the shear flow are horizontal lines,  the $x_1$-direction is sometimes referred to as the {\em streamwise} direction and the $x_2$-direction as the {\em spanwise} direction.

The ACHE is then given by the following equation:
\begin{equation} 
\label{eqn: sheareq}
\partial_t c+v(x_2) \partial_{x_1} c+\mu\nu\Delta^2 c=\nu \Delta(c^3-c).
\end{equation} 
In studying enhanced dissipation and mixing, it is convenient to indicate explicitly the amplitude of the velocity through a parameter $A>0$, that is, replace $\bv$ with $A \bv$. However,  a time reparameterization shows that $A$ corresponds to $\nu=1/A$ in \eqref{eqn:standard_CH}. 
Without loss of generality, we have therefore set $A=1$.
Our main result, Theorem \ref{thm:main}, validates rigorously the observed long-time dynamics in the two-dimensional
ACHE under the action of a shear flow in the following sense, at least for well-prepared data. We  specify the class of admissible initial data below (see \eqref{eqn:smallinitial}.  Informally, we show that, provided the parameter $\nu$ is small enough (equivalently, the shear strength is large enough), depending on the parameter $\mu$ and the size of the initial data, then the solution of the two-dimensional ACHE converges asymptotically at large times to the solution of a one-dimensional CHE, obtained by projecting the full equation in the direction orthogonal to the shear in a suitable sense (see \eqref{eqn: eqpar}). More precisely, we decompose the solution into two parts, one is the projection onto the kernel of the advection operator in \eqref{eqn: sheareq}, the other is the projection onto its $L^2$-orthogonal complement, and prove that the latter component decays rapidly as $t\to\infty$.

The one-dimensionalization of the long-time dynamics of the ACHE with shear flow can be established exploiting the enhanced dissipation resulting from the combined effect of (hyper)diffusion and advection. By {\em enhanced dissipation}, we mean in this context that the linear advection-(hyper)diffusion operator $-(-\Delta)^j +\bv\cdot\nabla$, $j=1,2$,
acts on characteristic timescales that are faster than those of the diffusion alone. The characteristic timescale is defined as the smallest time it takes for the solution operator to reduce the size of the solution by a fixed fraction. It is convenient to use the energy norm, i.e., the $L^2$ norm, to measure the size of the solution and to fix the fraction to be one half.

With abuse of notation, we will refer to a velocity field as a flow. Throughout, we treat only the case of incompressible flows, that is, $\bv$ is divergence free. We also assume that functions on $\TT^2$ are mean-free, as the mean is preserved under the Cahn-Hilliard evolution. This assumptions removes the constants, which are always in the kernel of the operator in \eqref{eqn: sheareq}. 

Enhanced dissipation can be measured in term of rates of decay of the operator in time and in the parameter $\nu$. Intuitively, if the vector field $\bv$ satisfies certain conditions, the advection operator $\bv\cdot\nabla$ transfers energy to small scales, where it is damped more efficiently by diffusion, as it can be seen by taking the Fourier Transform.  This phenomenon has been well studied in the mathematics literature, especially for reaction-diffusion equations and quenching in combustion (see e.g. \cite{FKR06,KZ06} and references therein). There are several examples of vector fields $\bv$ that are so-called {\em dissipation enhancing} ({\em relaxation enhancing} for the steady case). Informally, an  incompressible flow is dissipation enhancing if the dissipation time can be made arbitrarily small provided the amplitude of the flow is large enough (keeping fixed the diffusion coefficient). For steady Lipschitz-continuous flows and the case of Laplace's operator, the seminal contribution of Constantin {\em et al.} \cite{CKRZ08} gives a spectral characterization of relaxation-enhancing flows, namely, the advection operator cannot have any eigenfunctions in the $L^2$ domain of the Laplacian. Any such relaxation-enhancing flow is also dissipation enhancing for the bi-Laplacian, since the domain of latter is contained in the domain of the former. For unsteady flows, it is known that sufficiently regular mixing flows are dissipation enhancing (we refer the reader to \cite{FI19} and references therein for the definition of mixing flows and a more in-depth discussion). In the case of shear flows, there is a large kernel for the advection operator and any enhancement can only occur if the kernel is projected out. Then the enhancement depends on the shear profile, in particular on its critical points and their order. In the literature, dissipation enhancement has been established by using different tools, from hypocoercivity estimates (we mention in particular \cite{ABN22,BCZ17,CZG23}) to probabilistic methods \cite{CZD21} to resolvent estimates \cite{coti2021global,He22,Wei21} (see also \cite{FMN23} for shear flows in a circular geometry). Similar results hold, under certain conditions, for cellular flows \cite{iyer2022quantifying,F2022dissipation}.

There is a deep connection between mixing, enhanced dissipation and other phenomena occurring in fluids and in the equations that model them. We mention here only the phenomenon of {\em Taylor dispersion}, which is an enhancement of the rate of spreading of a species in the direction of the shear flow (see again \cite{CZG23} and references therein), and the stability of linearized viscous and inviscid flows around shear flows with the associated mechanism of {\em inviscid damping} (among the several recent results, we mention only \cite{Jia23,CWZ23,MZ20,BS22,GNRS20}).
Indeed, the combined effect of transport and diffusion can be either  stabilizing or  destabilizing  with consequences for the global well-posedness and the long-time behaviors of solutions to the underlying PDE.  A main example of stabilization is in aggregation models, such as the Patlak-Keller-Segel model, for which addition of advection by a sufficiently strong flow prevents blow-up of solutions irrespective of the total mass \cite{He23,HTZ22,bedrossian2017suppression}. Similarly, addition of a strong enough flow leads to global existence for the Kuramoto-Sivashinsky equation, a model of front propagation in combustion \cite{FM22,coti2021global}. The work closest to our setting is that of \cite{hu2021global}, where the authors consider an \emph{unstable} version of Cahn--Hilliard equation, characterized by the opposite sign in the cubic nonlinear term relative to \eqref{eqn:standard_CH}, thus lacking the energy structure \eqref{eqn: energy functional}.  Adding advection by a strong shear flow, they prove that, if the coefficient of the cubic term is sufficiently small and if the average of the initial data in the flow direction is likewise small, then there exists a global-in-time solution in $L^2$ even with a large initial energy. In contrast, it has been proved that without advection the equation can blow up in finite time if the initial energy is large \cite{ES86}. When using a unidirectional flow to enhance dissipation, one must assume that the linearized operator has growing modes only in the direction of the shear, as in \cite{coti2021global,hu2021global}.
A destabilizing effect of transport can be seen, for example, in another important phenomenon in fluids, namely, {\em anomalous dissipation} of energy in scalar turbulence \cite{BedrossianEtAl2022,ColomboEtAl2023,keefer2024,DrivasEtAl2022,yoneda2024,CheskidovLuo2024,armstrongVicol2025,huysmansTiti2025,elgindiLiss2024}. 

\subsection{Main result}
In order to present our main result, Theorem \ref{thm:main}, we need to introduce some preliminary concepts and some notation. We begin with defining the spatial average of any function $f\in L^2(\T^2)$: 
\begin{equation}
\langle f \rangle(t):=\int_{\T^2} f(t, x_1, x_2)dx_1 d x_2,
\end{equation}
where we used that we work with a torus of unitary size, for simplicity. Since the average is preserved under the CHE evolution, we can assume without loss of generality that $\langle c_0 \rangle=0$, which implies then $\langle c(t)\rangle\equiv 0$ for all $t>0$. For convenience, we denote the subspace of $L^2(\TT^2)$ of mean-zero functions by $\rL^2(\TT^2)$.

Then, we decompose the solution into its $\rL^2$-projection  onto the kernel of the advection operator $v(x_2) \partial_{x_1}$ and the projection onto its $\rL^2$-orthogonal complement.
To this end, we decompose a function $f\in L^\infty([0,T);\rL^2(\TT^2))$ into 
\begin{subequations}
\begin{align}
 f_{\parallel} (t,x_2)&:=\int_{\T} f(t,x_1,x_2)dx_1,    \label{eqn:decomposition_fa} \\
 f_{\notparallel}(t,x_1,x_2)&:=f(t,x_1,x_2)-f_{\parallel}(t,x_2). \label{eqn:decomposition_fb}
\end{align}
\end{subequations}
For the solution $c$, $c_{\parallel}$ informally represents the projection across the flow lines of $\bv$, which are horizontal lines, hence the notation. We stress that $f_{\parallel}$ depends only on $x_2$, while $f_{\notparallel}$ depends on both $x_1$ and $x_2$.

The shear profile is assumed not to vanish identically on an set of positive measure, otherwise no dissipation enhancement occurs, but it can have isolated zeros, which corresponds to stagnation points for the flow. Then we define the {\em order} of a critical point $x_0$ of a function $v(x)$ as the smallest integer $m\geq 2$
such that derivatives of order $m$ at $x_0$ are not all zero, that is, at most $m-1$ derivatives vanish at $x_0$. For a monotone shear profile on $\RR^2$ or a channel, for instance,  $m=1$. This definition is consistent with the classical one in Morse theory for instance, if applied to the flow map, instead of the velocity vector field. On the torus, regular shear flow cannot be monotone, hence $m\geq 2$.

Our main result is the following theorem, which implies that $c_{\parallel}(x_2)$ converges  asymptotically for large times to the solution of the one-dimensional Cahn-Hilliard equation (without advection) in the $x_2$ variable, provided the data is well prepared. The proof relies on  enhanced dissipation for the positive operator:
\begin{equation} \label{20210512eq111}
H_{\nu}:=\mu\nu \Delta^2+v(x_2) \partial_{x_1},
\end{equation} 
when $v$ has finitely many critical points of order at most $m\geq 2$. This enhancement can be characterized in terms of rates of decay for the solution operator of the associated advection-(hyper)diffusion equation, $e^{-t\,H_\nu}$, $t\geq 0$, which forms a strongly continuous semigroup in $\rL^2(\TT^2)$. We will use the following decay estimate (see \cite{BCZ17,CDE20,Wei21}):
\begin{equation} \label{20210512eq01}
\|e^{-t H_{\nu}} g_{\notparallel}\|_{L^2} \le 5\,e^{-\lambda_{\nu} t} \|g_{\notparallel}\|_{L^2}, \quad t>0,
\end{equation}
where 
\begin{equation} \label{eq:LambdanuDef}
    \lambda_{\nu}:=\delta_0 \nu^{\frac{2m}{2m+1}},
\end{equation}
with $\delta_0>0$ a constant that depends on $\mu$ and $m$, but is independent of $\nu$.

Throughout we use standard notation for spaces, e.g., $H^s(\TT^2)$, $s\in \RR$, denotes $L^2$-based Sobolev spaces, while $W^{k,\infty}(\TT^2)$, $k\in \NN$, denotes the Lipschitz space of order $k$. $C$ denotes a generic positive constant that may change line by line and, in particular, is independent of $\nu$. We also denote the $L^2$-inner product with $\left<\cdot, \cdot\right>$. We emphasize that the parameter $\mu>0$ is fixed from now on, though it can be taken to be small, as typical in applications. Thus the generic constant $C$, as well as the  parameter $\delta_0$ in the formula for the rate of enhanced dissipation $\lambda_\nu$ in \eqref{eq:LambdanuDef}, may depend on $\mu$. In some instances, for instance in \eqref{eqn:4estimates}, we explicitly write down the dependence on $\mu$ for clarity. The dependence on $\nu$ does not change in these instances.

\begin{thm}
\label{thm:main}
Let $v: \RR \rightarrow \mathbb{R}$ be a $1$-periodic function in $C^{m+1}(\TT)$ with a finite number of critical points of order at most $m\geq 2$. Let $c_0\in H^2(\T^2)\cap \rL^2(\TT^2)$, and let $c$ be the unique strong solution to \eqref{eqn: sheareq} with initial data $c_0$. There exists $0<\nu_0<1$ depending on $\|\partial_{x_2}v\|_{L^\infty}$, $\|(c_0)_{\parallel}\|_{L^2}$, $\|\partial_{x_2}(c_0)_{\parallel}\|_{L^2}$, $\|(c_0)_{\notparallel}\|_{L^2}$, $\|\partial_{x_2}(c_0)_{\notparallel}\|_{L^2}$ and $\mu$, but independent of $ \|\partial_{x_1}\left(c_0\right)_{\notparallel}\|_{L^2}$, such that, if  $\nu<\nu_0$ and $c_0$ in addition satisfies
\begin{equation}
\label{eqn:smallinitial}
    \|\partial_{x_1}\left(c_0\right)_{\notparallel}\|_{L^2}
    \leq \min{\left\{\lambda_{\nu},1\right\}},
\end{equation}
 then for any $t>0$,
\begin{equation}
\label{eqn:expon-decay}
\|c_{\notparallel}(t)\|_{L^2} \le 20e^{-\frac{\lambda_\nu t}{4}} \|c_{\notparallel}(0)\|_{L^2},   
\end{equation}
where $c_{\notparallel}$ is defined as in \eqref{eqn:decomposition_fb}.
\end{thm}

\begin{rem} \label{rem:example}
Several comments are in order:
  \leavevmode\newline 
\begin{enumerate}[label=(\alph*)]
\item The threshold $\nu_0$ is defined as the minimum of five quantities, $\nu_{0j}$ (with $j=1,\dots,5$), which appear in Lemmas~\ref{lem:B2}--\ref{lem:phipreB1} later in the article. Each $\nu_{0j}$ can be computed explicitly in terms of the norms of $c_0$, $v$, and their derivatives. Consequently the dependence of $\nu_0$ on these quantities is explicit. For the reader's convenience, we list in Table~\ref{chart:name_posi} below where each quantity is introduced.
\color{blue}
\begin{table}[ht]
\caption{Where the threshold values $\nu_{0,j}$ are introduced.}
\begin{center}
\begin{tabular}{ |c|c| } 
 \hline
 Constant name & Where defined \\ 
 \hline
 $\nu_{01}$ & Lemma \ref{lem:B2}  \\ 
 \hline
 $\nu_{02}$ & Lemma \ref{lem:L2control} \\ 
 \hline
 $\nu_{03}$ & Lemma \ref{lem:preB1} \\ 
 \hline
 $\nu_{04}$ & Lemma \ref{lem:phiL2control} \\ 
 \hline
 $\nu_{05}$ & Lemma \ref{lem:phipreB1} \\
 \hline
\end{tabular}
\end{center}
\label{chart:name_posi}
\end{table}
\color{black}

\medskip

\item We provide a simple example of initial data $c_0$ satisfying the  hypotheses of Theorem \ref{thm:main}. 
We let
\begin{equation*}
    c_0(x_1,x_2)=\sin{(2\pi x_2)}+\epsilon \sin{(2\pi x_1)}\sin{(2\pi x_2)},\quad 
    0<\epsilon\leq 1.
\end{equation*}
Since
\begin{align*}
 &(c_0)_{\parallel}=\sin{(2\pi x_2)},\,(c_0)_{\notparallel}=\epsilon \sin{(2\pi x_1)}\sin{(2\pi x_2)},\,\\ 
 &\partial_{x_2}(c_0)_{\parallel}=2\pi \cos{(2\pi x_2)}, \partial_{x_1}(c_0)_{\notparallel}=2\pi\epsilon\cos{(2\pi x_1)}\sin{(2\pi x_2)},\, \\
 &\partial_{x_2}(c_0)_{\notparallel}=2\pi\epsilon\sin{(2\pi x_1)}\cos{(2\pi x_2)},
\end{align*}
the norms $\|(c_0)_{\parallel}\|_{L^2}$, $\|\partial_{x_2}(c_0)_{\parallel}\|_{L^2}$, $\|(c_0)_{\notparallel}\|_{L^2}$, $\|\partial_{x_2}(c_0)_{\notparallel}\|_{L^2}$ can be computed explicitly, and so can $\nu_0$ . Hence, for $\epsilon$ small enough, the condition \eqref{eqn:smallinitial} holds and Theorem \ref{thm:main} can be applied to obtain \eqref{eqn:expon-decay}.

\medskip

\item Our results are in two space dimensions. It is natural to ask whether similar results hold in three space dimensions. In \cite{feng2022dissipation}, the authors show that a planar helical flow in $\TT^3$ can simultaneously enhance dissipation in two directions, thus reducing the dynamics of the three-dimensional advective Keller--Segel equation to an effective one-dimensional dynamics. We expect an analogous mechanism to hold for  the three-dimensional Cahn--Hilliard equation, provided that the initial data satisfies a smallness condition  on the component in the orthogonal complement of the advection operator akin to \eqref{eqn:smallinitial}. A complete analysis, however, will involve substantial technical details and is left for future work.

\item Condition \eqref{eqn:smallinitial} is not merely a technical assumption; rather it is justified, at least informally, by simulations as well. Numerical experiments show that, when the Cahn–Hilliard equation is driven by an external shear flow, solutions eventually form striped patterns, but the number of stripes depends on the shear strength (see Figure~5 in \cite{LDEBH13}). Stronger shear flows produce more stripes, meaning the flow only {\it partially stabilizes} the solution and leads to different steady states of the one-dimensional Cahn-Hilliard equation. This phenomenon suggests that shear flows can only suppress transverse variations below a certain scale, which mathematically corresponds precisely to the smallness assumption \eqref{eqn:smallinitial}  with threshold scale given by $\lambda_\nu$. Our results do not fully explain how the flow intensity relates to the final stripe count. This is an interesting question that we plan to investigate both numerically as well as analytically  in future work.
\end{enumerate}
\end{rem}

The rest of the paper is organized as follows. In Section \ref{sec:preliminary}, we recall some basic results about CHE and establish some preliminary estimates, which will be used for a bootstrap argument that is central to the proof of our main result. Section \ref{sec:Boundness of u_par} contains the proof of a key estimate, which gives a uniform bound on the derivative of the  projection $c_{\parallel}$ onto the kernel of the advection operator. By energy estimates, in turn a control on this term allows to close a Gr\"onwall estimate on the orthogonal complement $c_{\notparallel}$, which is crucial for the bootstrap argument. Indeed, in Section \ref{sec:Bootstrap estimates}, we are able to obtain bootstrap estimates under corresponding bootstrap assumptions using the results of the previous sections and conclude the proof of Theorem \ref{thm:main}. 

\subsection*{Acknowledgments}: The authors thank Michele Coti Zelati, Gautam Iyer, Yao Yao, and Zhifei Zhang for their helpful insight. Yu Feng was partially supported by the National Key R\&D Program of China, Project Number 2021YFA1001200, and the NSFC Youth program, Grant Number 12501669. Yuanyuan Feng was partially supported by National Key Research and Development Program of China (2022YFA1004401), NSFC 12301283, Shanghai Sailing program 23YF1410300, Science and Technology Commission of Shanghai Municipality (22DZ2229014), and Jiangsu Provincial Scientific Research Center of Applied Mathematics under Grant No. BK20233002. Anna Mazzucato thanks the Mathematics Department at Milan University, Milan, Italy, for their hospitality during a visit, when part of this work was conducted. She was partially supported by the US National Science Foundation under grants DMS-1615457, DMS-1909103, DMS-2206453 and  Simons Foundation Grant 1036502. Xiaoqian Xu was partially supported by National Key R\&D Program of China 2021YFA1001200, the National Science Foundation of China 12101278 and Kunshan Shuangchuang Talent Program kssc202102066.

\section{preliminary}
\label{sec:preliminary}
We devote this section to establishing some elementary estimates, which further induce the \emph{bootstrap assumptions} for \eqref{eqn: sheareq}. The key observation is that  the decay of the free energy yields a good enough control on derivatives of $c_{\notparallel}$.

We begin by recalling the well-posedness of strong solutions to the advective Cahn-Hilliard equation in \cite{FFIT19}.
\begin{thm}
\label{thm: local strong CH}
Let $\mu,\nu>0$, $v \in L^{\infty}([0,\infty);W^{2,\infty}(\T^2))$ and $c_0\in H^2(\T^2)$. There exists a unique strong solution to \eqref{eqn: sheareq} in the space
\begin{equation*}
    c(t,x)\in L^{2}_{\text{loc}}([0,\infty);H^4(\T^2))\cap L^{\infty}_{\text{loc}}([0,\infty);H^2(\T^2))\cap H^{1}_{\text{loc}}([0,\infty);L^2(\T^2)).
\end{equation*}
\end{thm}

Throughout, then $c$ denotes the unique strong solution  to \eqref{eqn: sheareq} with $c_0\in H^2(\T^2)\cap \rL^2(\TT^2)$. 
Given the regularity of $c$, both $c$ and its derivatives enjoy an energy identity.
As mentioned in the Introduction, we decompose $c$ into $c_{\parallel}$ and $c_{\notparallel}$ as in \eqref{eqn:decomposition_fa} and \eqref{eqn:decomposition_fb}, respectively. Thus, $c_{\parallel}$ and $c_{\notparallel}$ satisfy the following system of coupled equations: 
\begin{subequations} \label{eq:coupled}
\begin{equation} 
\label{eqn: eqpar}
\partial_t c_{\parallel} + \mu\nu \partial_{x_2}^4 c_{\parallel}=-\nu\partial_{x_2}^2 c_{\parallel} + \nu\int_{\T}\partial_{x_2}^2(c_{\parallel}+c_{\notparallel})^3 dx_1,
\end{equation}
\begin{align}
& \partial_t c_{\notparallel}+v(x_2)\partial_{x_1} c_{\notparallel}+\mu\nu \Delta^2 c_{\notparallel} = \label{eqnotpar} \\
-\nu\Delta c_{\notparallel}+\nu\Delta &\left(3 c_{\parallel}^2 c_{\notparallel}+3c_{\parallel} c_{\notparallel}^2 +c_{\notparallel}^3\right)-\nu\int_{\T}\partial_{x_2}^2\left(3c_{\parallel} c_{\notparallel}^2+c_{\notparallel}^3\right) dx_1, \nonumber
\end{align} 
\end{subequations}
where we used that $c_{\parallel}$ is independent on $x_1$ and that $\langle c_{\notparallel}\rangle_{x_1}=0$ to obtain \eqref{eqnotpar}. 
We remark that, if  $c_{\notparallel}$ vanishes, then  \eqref{eqn: eqpar} reduces to the one-dimensional Cahn-Hilliard equation in the spanwise direction.

We tackle \eqref{eqnotpar} first, deriving a standard energy estimate:
\begin{align}
&\frac{d}{dt}\|c_{\notparallel}\|_{L^2}^2+2\mu\nu\|\Delta c_{\notparallel}\|_{L^2}^2 \nonumber \\
&=2\nu\|\nabla c_{\notparallel}\|_{L^2}^2+2\nu\left<\Delta\left(3 c_{\parallel}^2 c_{\notparallel}+3c_{\parallel} c_{\notparallel}^2+c_{\notparallel}^3\right), c_{\notparallel}\right> \nonumber \\
&\quad -2\nu\left<\partial_{x_2}^2\int_{\T}\left(3c_{\parallel} c_{\notparallel}^2+c_{\notparallel}^3\right)  dx_1, c_{\notparallel}\right> \nonumber \\
&=2\nu\|\nabla c_{\notparallel}\|_{L^2}^2-6\nu\|c_{\notparallel}\nabla c_{\notparallel}\|_{L^2}^2+2\nu\left<\Delta\left(3 c_{\parallel}^2 c_{\notparallel}+3c_{\parallel} c_{\notparallel}^2\right), c_{\notparallel}\right> \nonumber  \\
&\quad -2\nu\left<\partial_{x_2}^2\int_{\T}\left(3c_{\parallel} c_{\notparallel}^2+c_{\notparallel}^3\right)  dx_1, c_{\notparallel}\right>\nonumber \\
&\leq 2\nu\Bigg(\|\nabla c_{\notparallel}\|^2_{L^2}+\left\vert\left<\left(3 c_{\parallel}^2 c_{\notparallel}+3c_{\parallel} c_{\notparallel}^2\right),\Delta c_{\notparallel}\right>\right\vert \nonumber  \\
&\quad\qquad+\left\vert\left<\int_{\T}\left(3c_{\parallel} c_{\notparallel}^2+c_{\notparallel}^3 \right) dx_1,\partial_{x_2}^2 c_{\notparallel}\right>\right\vert\Bigg). \label{eqn:L2first_estimate}
\end{align}
Integrating by parts and applying H\"older's inequality, we can bound each term on the right-hand side as follows:
\begin{align}
&\|\nabla c_{\notparallel}\|^2_{L^2}\leq \|\Delta c_{\notparallel}\|_{L^2}\|c_{\notparallel}\|_{L^2}, \nonumber \\
&\left\vert\left<c_{\parallel}^2 c_{\notparallel},\Delta c_{\notparallel}\right>\right\vert 
\leq \|c_{\parallel}\|_{L^\infty}^2\|c_{\notparallel}\|_{L^2}\|\Delta c_{\notparallel}\|_{L^2},  \label{eqn:1d-GNieq-H1}\\
&\left\vert\left<c_{\parallel} c_{\notparallel}^2,\Delta c_{\notparallel}\right>\right\vert 
\leq \|c_{\parallel}\|_{L^\infty}\|c_{\notparallel}\|_{L^4}^2\|\Delta c_{\notparallel}\|_{L^2}, \nonumber \\
&\left\vert\left<c_{\notparallel}^3,\Delta c_{\notparallel}\right>\right\vert \leq \|c_{\notparallel}\|_{L^6}^3\|\Delta c_{\notparallel}\|_{L^2}. \nonumber 
\end{align}
We now recall well-known Gagliardo-Nirenberg-Moser estimates (see e.g. \cite[Chapter 13]{TaylorPDEIII2023} and references therein).  In particular, in one dimension, using also Poincar\'e's inequality for mean-free functions, we have that 
\begin{equation} \label{1dGNineqn}
 \|f\|_{L^\infty(\TT)}\leq C\|\nabla f\|_{L^2(\TT)}^{1/2}\|f\|_{L^2(\TT)}^{1/2}\leq C\|\nabla f\|_{L^2(\TT)},
\end{equation}
while in dimension two one has that
\begin{align}
\|f\|_{L^4(\TT^2)}&\leq C\|\nabla f\|_{L^2(\TT^2)}^{1/2}\|f\|_{L^2(\TT^2)}^{1/2},\nonumber \\
 \|f\|_{L^4(\TT^2)}&\leq C\|\Delta f\|_{L^2(\TT^2)}^{1/4}\|f\|_{L^2(\TT^2)}^{3/4}, \label{eqn:three_GN_ineq} \\
 \|f\|_{L^6(\TT^2)}&\leq C\|\nabla f\|_{L^2(\TT^2)}^{2/3}\|f\|_{L^2(\TT^2)}^{1/3}. \nonumber
\end{align}
Using the inequalities above and Young's inequality in \eqref{eqn:1d-GNieq-H1} then gives:
\begin{align}
&\|\nabla c_{\notparallel}\|^2_{L^2}\leq \frac{1}{8}\mu\|\Delta c_{\notparallel}\|_{L^2}^2 + \frac{C}{\mu}\|c_{\notparallel}\|_{L^2}^2, \nonumber \\
&\left\vert\left<c_{\parallel}^2 c_{\notparallel},\Delta c_{\notparallel}\right>\right\vert 
\leq C\|\partial_{x_2} c_{\parallel}\|_{L^2}^2\|c_{\notparallel}\|_{L^2}\|\Delta c_{\notparallel}\|_{L^2} \nonumber   \\\nonumber
&\qquad\qquad\qquad\quad\leq \frac{1}{24}\mu\|\Delta c_{\notparallel}\|_{L^2}^2+\frac{C}{\mu}\|\partial_{x_2} c_{\parallel}\|_{L^2}^4\|c_{\notparallel}\|_{L^2}^{2},\\ \nonumber 
&\left\vert\left<c_{\parallel} c_{\notparallel}^2,\Delta c_{\notparallel}\right>\right\vert 
\leq C\|\partial_{x_2} c_{\parallel}\|_{L^2}\|c_{\notparallel}\|_{L^2}^{3/2}\|\Delta c_{\notparallel}\|_{L^2}^{3/2}\\\label{eqn:4estimates}
&\qquad\qquad\qquad\quad\leq \frac{1}{24}\mu\|\Delta c_{\notparallel}\|_{L^2}^2 + \frac{C}{\mu} \|\partial_{x_2} c_{\parallel}\|_{L^2}^4\|c_{\notparallel}\|_{L^2}^{6},\\ \nonumber
&\left\vert\left<c_{\notparallel}^3,\Delta c_{\notparallel}\right>\right\vert \leq C \|c_{\notparallel}\|_{L^2}\|\nabla c_{\notparallel}\|_{L^2}^2\|\Delta c_{\notparallel}\|_{L^2}\\\nonumber
&\qquad\qquad\qquad\quad\leq \frac{1}{8}\mu\|\Delta c_{\notparallel}\|_{L^2}^2+\frac{C}{\mu}\|\nabla c_{\notparallel}\|_{L^2}^4\|c_{\notparallel}\|_{L^2}^2. \nonumber
\end{align}
Consequently,  \eqref{eqn:L2first_estimate} becomes
\begin{align}
\label{eqn: energy_estimateI}
    &\frac{d}{dt}\| c_{\notparallel}\|_{L^2}^2 + \mu\nu\|\Delta c_{\notparallel}\|_{L^2}^2 \leq 
    \\\nonumber  
    C\nu\big( \| c_{\notparallel}\|_{L^2}^2 &+ \|\partial_{x_2} c_{\parallel}\|_{L^2}^4\|c_{\notparallel}\|_{L^2}^{2} + \|\partial_{x_2} c_{\parallel}\|_{L^2}^4 \|c_{\notparallel}\|_{L^2}^6 + \|\nabla c_{\notparallel}\|_{L^2}^4\|c_{\notparallel}\|_{L^2}^2\big),
\end{align}
from which it follows that, for $0\leq s\leq t,$
\begin{align}
\label{eqn: energy_estimateII}
    &\mu\nu\int_s^t\|\Delta c_{\notparallel}(\tau)\|_{L^2}^2 d\tau
    \leq \| c_{\notparallel}(s)\|_{L^2}^2 + C \nu  \int_s^t \| c_{\notparallel}\|_{L^2}^2\, d\tau
    \\\nonumber
 + C\nu \int^t_s \big(\|\partial_{x_2} &c_{\parallel}\|_{L^2}^4  \|c_{\notparallel}\|_{L^2}^{2}+\|\partial_{x_2} c_{\parallel}\|_{L^2}^4 \|c_{\notparallel}\|_{L^2}^6 + \|\nabla c_{\notparallel}\|_{L^2}^4\|c_{\notparallel}\|_{L^2}^2\big)\, d\tau,
\end{align}
where we have neglected the positive term $\|c_{\notparallel}(t)\|_{L^2}$ on the left-hand side.

Next, we estimate $\|\partial_{x_1}c_{\notparallel}\|_{L^2}$ and $\|\partial_{x_2}c_{\notparallel}\|_{L^2}$ separately. For simplicity, we denote $\phi:=\partial_{x_1}c_{\notparallel}$ and $\psi:=\partial_{x_2}c_{\notparallel}$. Differentiating both sides of  \eqref{eqnotpar} with respect to $x_1$ gives the equation for $\phi$:
\begin{align}
    &\partial_t\phi+v(x_2)\partial_{x_1}\phi+\mu\nu\Delta^2\phi\\
    &\qquad\qquad\qquad=-\nu\Delta\phi +\nu\Delta\left(\partial_{x_1}\left(3 c_{\parallel}^2 c_{\notparallel}+3 c_{\parallel}c_{\notparallel}^2+c_{\notparallel}^3\right)\right), \label{eqn:phieqn}
\end{align}
from which we obtain the energy identity for $\phi$ in a standard way, multiplying  both sides of the equation by $2\phi$ and integrating by parts: 
\begin{align*}
\frac{d}{dt}\|\phi\|_{L^2}^2+2\mu\nu\|\Delta\phi\|_{L^2}^2
&=2\nu\left<-\phi+\partial_{x_1}\left(3 c_{\parallel}^2 c_{\notparallel}+3 c_{\parallel}c_{\notparallel}^2+c_{\notparallel}^3\right), \Delta\phi\right>\\
&=2\nu\left<-\phi+3 c_{\parallel}^2 \phi+6 c_{\parallel}c_{\notparallel}\phi+ 3c_{\notparallel}^2\phi, \Delta\phi\right>.
\end{align*}
Now, we can readily estimate some of the terms on the right as follows:
\begin{align*}
&\vert\left<-\phi,\Delta\phi\right>\vert\leq\frac{\mu}{8}\|\Delta\phi\|_{L^2}^2+\frac{C}{\mu}\|\phi\|_{L^2}^2,\\
&\left\vert\left<c_{\parallel}^2 \phi,\Delta\phi\right>\right\vert
\leq\frac{\mu}{24}\|\Delta\phi\|_{L^2}^2+\frac{C}{\mu}\Vert c_{\parallel}^2\phi\Vert_{L^2}^2\leq\frac{\mu}{24}\|\Delta\phi\|_{L^2}^2+C\Vert c_{\parallel}\Vert_{L^{\infty}}^4\Vert\phi\Vert_{L^2}^2\\\nonumber
&\qquad\qquad\qquad\leq\frac{\mu}{24}\|\Delta\phi\|_{L^2}^2+C\Vert \partial_{x_2 }c_{\parallel}\Vert_{L^2}^4\Vert\phi\Vert_{L^2}^2,
\end{align*}
using \eqref{1dGNineqn} on  $\Vert c_{\parallel}\Vert_{L^{\infty}}$.
To bound the trilinear term, in addition to \eqref{1dGNineqn}, we employ the following standard Gagliardo-Nirenberg-Moser estimate for mean-zero-functions in dimension two (see e.g. \cite{TaylorPDEIII2023} again) on  $\Vert c_{\notparallel}\Vert_{L^{\infty}}$:
\begin{align}
     \|f\|_{L^\infty(\TT^2)}&\leq C\|\Delta f\|_{L^2(\TT^2)}^{1/2}\|f\|_{L^2(\TT^2)}^{1/2}; \label{eq:2DGagliardo} \\
    \|f\|_{L^\infty(\TT^2)}&\leq C\|\Delta f\|_{L^2(\TT^2)}^{1/3}\|f\|_{L^4(\TT^2)}^{2/3}\leq C\|\Delta f\|_{L^2(\TT^2)}^{1/3}\|\nabla f\|_{L^2(\TT^2)}^{1/3}\|f\|_{L^2(\TT^2)}^{1/3}. \nonumber
\end{align}     
It follows that
\begin{align*}
&\left\vert\left< c_{\parallel}c_{\notparallel}\phi,\Delta\phi\right>\right\vert
\leq\frac{\mu}{48}\|\Delta\phi\|_{L^2}^2+\frac{C}{\mu}\Vert c_{\parallel}c_{\notparallel}\phi\Vert_{L^2}^2\\\nonumber
&\qquad\qquad\qquad\quad\leq \frac{\mu}{48}\|\Delta\phi\|_{L^2}^2+C\Vert c_{\parallel}\Vert_{L^{\infty}}^2 \Vert c_{\notparallel}\Vert_{L^{\infty}}^2 \Vert\phi\Vert_{L^2}^2\\\nonumber
&\qquad\qquad\qquad\quad\leq  \frac{\mu}{48}\|\Delta\phi\|_{L^2}^2+C\Vert \partial_{x_2} c_{\parallel}\Vert_{L^2}^2\Vert\Delta c_{\notparallel}\Vert_{L^2}\Vert c_{\notparallel}\Vert_{L^2}\Vert\phi\Vert_{L^2}^2,\\
&\left\vert\left<c_{\notparallel}^2\phi, \Delta\phi\right>\right\vert
\leq\frac{\mu}{24}\|\Delta\phi\|_{L^2}^2+\frac{C}{\mu}\Vert c_{\notparallel}^2\phi\Vert_{L^2}^2\leq\frac{\mu}{24}\|\Delta\phi\|_{L^2}^2+C\Vert c_{\notparallel}\Vert_{L^{\infty}}^4\Vert\phi\Vert_{L^2}^2\\
&\qquad\qquad\qquad\leq\frac{\mu}{24}\|\Delta\phi\|_{L^2}^2+C\Vert \Delta c_{\notparallel}\Vert_{L^2}^{4/3}\Vert \nabla c_{\notparallel}\Vert_{L^2}^{4/3}\Vert c_{\notparallel}\Vert_{L^2}^{4/3}\Vert\phi\Vert_{L^2}^2.
\end{align*}
By combining the estimates above, we finally have that
\begin{align}
\label{eqn:phiL2}
\frac{d}{dt}\|\phi\|_{L^2}^2
&\leq C\nu\Big(1+\Vert \partial_{x_2} c_{\parallel}\Vert_{L^2}^4+\Vert \partial_{x_2} c_{\parallel}\Vert_{L^2}^2\Vert\Delta c_{\notparallel}\Vert_{L^2}\Vert c_{\notparallel}\Vert_{L^2}\\\nonumber
&\quad +\Vert \Delta c_{\notparallel}\Vert_{L^2}^{4/3}\Vert\nabla c_{\notparallel}\Vert_{L^2}^{4/3}\Vert c_{\notparallel}\Vert_{L^2}^{4/3}\Big)\Vert\phi\Vert_{L^2}^2.
\end{align}
Next, we turn to  $\psi$. Differentiating equation \eqref{eqnotpar} with respect to $x_2$ gives:
\begin{align}
&\partial_t\psi+v(x_2)\partial_{x_1}\psi+\left(\partial_{x_2}v\right)\phi+\mu\nu\Delta^2\psi= \nonumber \\\nonumber
&\quad-\nu\Delta\psi+\nu\Delta\left(\partial_{x_2}\left(3 c_{\parallel}^2 c_{\notparallel}+3 c_{\parallel}c_{\notparallel}^2+c_{\notparallel}^3\right)\right)-\nu\partial_{x_2}^3\left(\int_{\TT}\left(3c_{\parallel}c_{\notparallel}^2+c_{\notparallel}^3\right)d x_1\right).
\end{align}
Again, a standard integration by parts gives the energy identity for $\psi$:
\begin{align}
&\frac{d}{dt}\|\psi\|_{L^2}^2+2\mu\nu\|\Delta\psi\|_{L^2}^2 \nonumber \\\nonumber
&=-2\left<\left(\partial_{x_2}v\right)\phi,\psi\right>-2\nu\left<\psi,\Delta\psi\right>+6\nu\left<c_{\parallel}^2\psi,\Delta\psi\right>+12\nu\left<c_{\parallel}c_{\notparallel}\partial_{x_2}c_{\parallel},\Delta\psi\right>\\\nonumber
&\quad+6\nu\left<c_{\notparallel}^2\partial_{x_2}c_{\parallel},\Delta\psi\right>+12\nu\left<c_{\parallel}c_{\notparallel}\psi,\Delta\psi\right>+6\nu\left<c_{\notparallel}^2\psi,\Delta\psi\right>\\\nonumber
&\quad-6\nu\left<\int_{\TT}\left(c_{\notparallel}^2\partial_{x_2}c_{\parallel}\right)dx_1,\partial_{x_2}^2\psi\right>
-12\nu\left<\int_{\TT}\left(c_{\parallel}c_{\notparallel}\psi\right)dx_1,\partial_{x_2}^2\psi\right>\\\nonumber
&\quad-6\nu\left<\int_{\TT}\left(c_{\notparallel}^2\psi\right)dx_1,\partial_{x_2}^2\psi\right>
\end{align}
Differently than in the energy identity for $\phi$, in the identity above the first term on the right-hand side does not contain the small parameter $\nu$. However, it only depends on $\phi$ and we will show that $\phi$ decays exponentially fast. Thus, under the smallness assumption \eqref{eqn:smallinitial}, one can obtain a uniform bound for $\Vert\psi\Vert_{L^2}$. 
To prove this result, we estimate each term on the right-hand side of the energy identity for $\psi$. Again, we will use \eqref{1dGNineqn} for $\Vert c_{\parallel}\Vert_{L^{\infty}}$, and \eqref{eq:2DGagliardo} 
for $\Vert c_{\notparallel}\Vert_{L^{\infty}}$. The estimates are tedious, but standard, giving
\begin{align*}
|\left<\left(\partial_{x_2}v\right)\phi,\psi\right>| &\leq\Vert\partial_{x_2}v\Vert_{L^{\infty}}\Vert\phi\Vert_{L^2}\Vert\psi\Vert_{L^2},
\end{align*}
\begin{align*}
|\left<\psi,\Delta\psi\right>| &\leq\frac{\mu}{18}\|\Delta\psi\|_{L^2}^2+\frac{C}{\mu}\Vert\psi \Vert_{L^2}^2\leq \frac{\mu}{18}\|\Delta\psi\|_{L^2}^2+C\Vert\nabla c_{\notparallel} \Vert_{L^2}^2\\
&\leq \frac{\mu}{18}\|\Delta\psi\|_{L^2}^2+C\Vert\Delta c_{\notparallel} \Vert_{L^2}^2,
\end{align*}
\begin{align*}
\left\vert \left<c_{\parallel}^2\psi,\Delta\psi\right>\right\vert 
&\leq \frac{\mu}{54}\Vert \Delta\psi\Vert_{L^2}^2+\frac{C}{\mu}\Vert c_{\parallel}^2\psi \Vert_{L^2}^2 \\
&\leq \frac{\mu}{54}\Vert \Delta\psi\Vert_{L^2}^2+C\Vert c_{\parallel}\Vert_{L^{\infty}}^4\Vert\psi \Vert_{L^2}^2\\\nonumber
&\leq\frac{\mu}{54}\Vert \Delta\psi\Vert_{L^2}^2+C\Vert c_{\parallel}\Vert_{L^{\infty}}^4\Vert\nabla c_{\notparallel} \Vert_{L^2}^2\\
&\leq\frac{\mu}{54}\Vert \Delta\psi\Vert_{L^2}^2+C\Vert \partial_{x_2}c_{\parallel}\Vert_{L^2}^4\Vert\Delta c_{\notparallel} \Vert_{L^2}\Vert c_{\notparallel} \Vert_{L^2},
\end{align*}
\begin{align*}
\left\vert\left<c_{\parallel}c_{\notparallel}\partial_{x_2}c_{\parallel},\Delta\psi\right>\right\vert
&\leq\frac{\mu}{108}\Vert \Delta\psi\Vert_{L^2}^2+\frac{C}{\mu}\Vert c_{\parallel}c_{\notparallel}\partial_{x_2}c_{\parallel}\Vert_{L^2}^2\\
&\leq\frac{\mu}{108}\Vert \Delta\psi\Vert_{L^2}^2+C\Vert c_{\parallel}\Vert_{L^\infty}^2\Vert c_{\notparallel}\Vert_{L^\infty}^2 \Vert\partial_{x_2} c_{\parallel}\Vert_{L^2}^2\\\nonumber
&\leq\frac{\mu}{108}\Vert \Delta\psi\Vert_{L^2}^2+C\Vert\partial_{x_2} c_{\parallel}\Vert_{L^2}^4\Vert \Delta c_{\notparallel}\Vert_{L^2}\Vert c_{\notparallel}\Vert_{L^2},
\end{align*}
\begin{align*}
\left\vert\left<c_{\notparallel}^2\partial_{x_2}c_{\parallel},\Delta\psi\right>\right\vert 
&\leq\frac{\mu}{54}\Vert \Delta\psi\Vert_{L^2}^2+\frac{C}{\mu}\Vert c_{\notparallel}^2\partial_{x_2}c_{\parallel}\Vert_{L^2}^2\\
&\leq\frac{\mu}{54}\Vert \Delta\psi\Vert_{L^2}^2+C\Vert\partial_{x_2} c_{\parallel}\Vert_{L^2}^2\Vert \Delta c_{\notparallel}\Vert_{L^2}^{4/3}\Vert \nabla c_{\notparallel}\Vert_{L^2}^{4/3}\Vert c_{\notparallel}\Vert_{L^2}^{4/3}\\
&\leq\frac{\mu}{54}\Vert \Delta\psi\Vert_{L^2}^2+C\left(\Vert\partial_{x_2} c_{\parallel}\Vert_{L^2}^2\Vert \Delta c_{\notparallel}\Vert_{L^2}^{2/3}\Vert \nabla c_{\notparallel}\Vert_{L^2}^{2/3}\Vert c_{\notparallel}\Vert_{L^2}^{2/3}\right)^{3/2}\\
&\quad+C\left(\Vert \Delta c_{\notparallel}\Vert_{L^2}^{2/3}\Vert \nabla c_{\notparallel}\Vert_{L^2}^{2/3}\Vert c_{\notparallel}\Vert_{L^2}^{2/3}\right)^{3}\\
&\leq\frac{\mu}{54}\Vert \Delta\psi\Vert_{L^2}^2+C\Vert\partial_{x_2} c_{\parallel}\Vert_{L^2}^3\Vert \Delta c_{\notparallel}\Vert_{L^2}\left(\Vert\phi\Vert_{L^2}+\Vert\psi\Vert_{L^2}\right)\Vert c_{\notparallel}\Vert_{L^2}\\
&\quad+C\Vert \Delta c_{\notparallel}\Vert_{L^2}^{2}\left(\Vert\phi\Vert_{L^2}^{2}+\Vert\psi\Vert_{L^2}^{2}\right)\Vert c_{\notparallel}\Vert_{L^2}^2,
\end{align*}
\begin{align*}
\left\vert\left<c_{\parallel}c_{\notparallel}\psi,\Delta\psi\right>\right\vert 
&\leq\frac{\mu}{108}\Vert \Delta\psi\Vert_{L^2}^2+\frac{C}{\mu}\Vert c_{\parallel}c_{\notparallel}\psi\Vert_{L^2}^2\\\nonumber
& \leq\frac{\mu}{108}\Vert \Delta\psi\Vert_{L^2}^2+C\Vert c_{\parallel}\Vert_{L^{\infty}}^2\Vert\psi\Vert_{L^2}^2\Vert c_{\notparallel}\Vert_{L^{\infty}}^2\\
&\leq\frac{\mu}{108}\Vert \Delta\psi\Vert_{L^2}^2+C\Vert\partial_{x_2} c_{\parallel}\Vert_{L^2}^2\Vert\psi\Vert_{L^2}^2\Vert \Delta c_{\notparallel}\Vert_{L^2}\Vert c_{\notparallel}\Vert_{L^2},
\end{align*}
\begin{align*}
\left\vert\left<c_{\notparallel}^2\psi,\Delta\psi\right>\right\vert 
&\leq\frac{\mu}{54}\Vert \Delta\psi\Vert_{L^2}^2+\frac{C}{\mu}\Vert c_{\notparallel}^2\psi\Vert_{L^2}^2\\\nonumber
&\leq\frac{\mu}{54}\Vert \Delta\psi\Vert_{L^2}^2+C\Vert \Delta c_{\notparallel}\Vert_{L^2}^2\Vert c_{\notparallel}\Vert_{L^2}^2\Vert\psi\Vert_{L^2}^2.
\end{align*}
We next note that Jensen's inequality and Tonelli's theorem imply that
\begin{align*}
\left\Vert f_{\parallel} \right\Vert_{L^2}
&=|\T|\left(\int_{\TT}\left\vert\frac{1}{|\T|}\int_{\TT} f(x_1,x_2)dx_1\right\vert^2 d x_2\right)^{1/2}\\
&\leq |\T|\left(\int_{\TT}\frac{1}{|\T|}\int_{\TT} |f(x_1,x_2)|^2dx_2 \,dx_1\right)^{1/2}\\
&\leq \sqrt{|\T|}\Vert f(x_1,x_2)\Vert_{L^2}.
\end{align*}
Thus, parallel to the last three estimates on the last page, we also have:
\begin{align*}
&\left\vert\left<\int_{\TT}\left(c_{\notparallel}^2\partial_{x_2}c_{\parallel}\right)dx_1,\partial_{x_2}^2\psi\right>\right\vert\\
&\leq \Vert \partial_{x_2}^2\psi\Vert_{L^2}\Vert c_{\notparallel}^2\partial_{x_2}c_{\parallel}\Vert_{L^2}\\
&\leq\frac{\mu}{54}\Vert \Delta\psi\Vert_{L^2}^2+C\Vert\partial_{x_2} c_{\parallel}\Vert_{L^2}^3\Vert \Delta c_{\notparallel}\Vert_{L^2}\left(\Vert\phi\Vert_{L^2}+\Vert\psi\Vert_{L^2}\right)\Vert c_{\notparallel}\Vert_{L^2}\\
&\quad+C\Vert \Delta c_{\notparallel}\Vert_{L^2}^{2}\left(\Vert\phi\Vert_{L^2}^{2}+\Vert\psi\Vert_{L^2}^{2}\right)\Vert c_{\notparallel}\Vert_{L^2}^2,
\end{align*}
\begin{align*}
&\left\vert\left<\int_{\TT}\left(c_{\parallel}c_{\notparallel}\psi\right)dx_1,\partial_{x_2}^2\psi\right>\right\vert\\
&\leq \Vert\partial_{x_2}^2\psi\Vert_{L^2}\Vert c_{\parallel}c_{\notparallel}\psi\Vert_{L^2}\\
&\leq \frac{\mu}{108}\Vert \Delta\psi\Vert_{L^2}^2+C\Vert\partial_{x_2} c_{\parallel}\Vert_{L^2}^2\Vert\psi\Vert_{L^2}^2\Vert \Delta c_{\notparallel}\Vert_{L^2}\Vert c_{\notparallel}\Vert_{L^2},
\end{align*}
\begin{align*}
&\left\vert\left<\int_{\TT}\left(c_{\notparallel}^2\psi\right)dx_1,\partial_{x_2}^2\psi\right>\right\vert\\
&\leq \Vert \partial_{x_2}^2\psi\Vert_{L^2}\Vert c_{\notparallel}^2\psi \Vert_{L^2}\\
&\leq \frac{\mu}{54}\Vert \Delta\psi\Vert_{L^2}^2+C\Vert\Delta c_{\notparallel}\Vert_{L^2}^2\Vert c_{\notparallel}\Vert_{L^2}^2\Vert\psi\Vert_{L^2}^2.
\end{align*}
By combining all the estimates above, we finally obtain
\begin{align}
\frac{d}{dt}\|\psi\|_{L^2}^2&\leq C^\ast\nu\left(\Vert\Delta c_{\notparallel}\Vert_{L^2}^2\Vert c_{\notparallel}\Vert_{L^2}^2+\Vert\partial_{x_2} c_{\parallel}\Vert_{L^2}^2\Vert\Delta c_{\notparallel}\Vert_{L^2}\Vert c_{\notparallel}\Vert_{L^2}\right)\Vert\psi\Vert_{L^2}^2 \nonumber \\
\nonumber
&+\left(\Vert\partial_{x_2}v\Vert_{L^{\infty}}\Vert\phi\Vert_{L^2}+C^\ast\nu\Vert\partial_{x_2} c_{\parallel}\Vert_{L^2}^3\Vert\Delta c_{\notparallel}\Vert_{L^2}\Vert c_{\notparallel}\Vert_{L^2}\right)\Vert\psi\Vert_{L^2}\\\nonumber
&+C^\ast\nu\Big(\Vert \Delta c_{\notparallel}\Vert_{L^2}^2+\Vert\partial_{x_2} c_{\parallel}\Vert_{L^2}^4\Vert \Delta c_{\notparallel}\Vert_{L^2}\Vert c_{\notparallel}\Vert_{L^2}\\
&\qquad\qquad+\Vert\partial_{x_2} c_{\parallel}\Vert_{L^2}^3\Vert\Delta c_{\notparallel}\Vert_{L^2}\Vert c_{\notparallel}\Vert_{L^2}\Vert\phi\Vert_{L^2}\\
&\qquad\qquad+\Vert\Delta c_{\notparallel}\Vert_{L^2}^2\Vert c_{\notparallel}\Vert_{L^2}^2\Vert\phi\Vert_{L^2}^2\Big),
\label{eqn:psiL2esti}
\end{align}
where $C^\ast$ is a constant that only depends on $\mu$.

We will prove Theorem \ref{thm:main} by using a bootstrap argument. We start by making the following \emph{bootstrap  assumptions} based on estimates  \eqref{eqn: energy_estimateI}, \eqref{eqn: energy_estimateII},  \eqref{eqn:phiL2}, and \eqref{eqn:psiL2esti}. For any $\nu>0$ and $0 \le s \le t \le t_0$, where $t_0$ is to be determined later:
\begin{enumerate}[label={\bf (H\arabic*)}, ref=\textcolor{black}{\bf (H\arabic*)}] \label{l:bootstrap_a}
    \item $\|c_{\notparallel}(t)\|_{L^2} \le 20e^{-\frac{\lambda_\nu(t-s)}{4}} \|c_{\notparallel}(s)\|_{L^2}$; 
    \label{i:bootstrap_a.1}

    \medskip
    
    \item $\mu\nu\int_{s}^{t}\|\Delta c_{\notparallel}(\tau)\|_{L^2}^2 d\tau \le 10\|c_{\notparallel}(s)\|_{L^2}^2$; 
    \label{i:bootstrap_a.2}

    \medskip
    
    \item  $\|\phi(t)\|_{L^2} \le 20e^{-\frac{\lambda_\nu(t-s)}{4}} \|\phi(s)\|_{L^2}$; 
    \label{i:bootstrap_a.3}

    \medskip

    \item  $\|\psi(t)\|_{L^2}^2\leq Y_0\cdot \exp{\left[C_2\frac{10}{\mu}\Vert c_{\notparallel}(0)\Vert_{L^2}^4+2\right]}:=M_0$;
    \label{i:bootstrap_a.4}
\end{enumerate}
The constant $Y_0$ above can be explicitly computed as 
\begin{align}
\label{eqn:Y0}
    Y_0&=2\Big(\|\psi(0)\|_{L^2}^2+C_1\left(\Vert c_{\notparallel}(0)\Vert_{L^2}^2+4\right)\Vert c_{\notparallel}(0)\Vert_{L^2}^2\\\nonumber
    &\qquad+\frac{\left(20 \Vert\partial_{x_2}v\Vert_{L^{\infty}}+1\right)^2}{4} \Big),
\end{align}
where $C_1=160000(\sqrt{\frac{10}{\mu}}+1)^2C^\ast$, $C_2=20C^\ast$ are constants depending only on $\mu$, and we recall $v$ is the given shear profile. Since $c$ is a strong solution, by the well-posedness of the ACHE assumptions \ref{i:bootstrap_a.1}--\ref{i:bootstrap_a.4} hold on a, possibly small, positive time interval. We then define $t_0$ to be the largest time such that the bootstrap assumptions hold on $[0,t_0]$.
\begin{rem}
\label{rmk1}
We observe that \ref{i:bootstrap_a.3}--\ref{i:bootstrap_a.4} and assumption \eqref{eqn:smallinitial} together yields
\begin{align}
\label{eqn:Z0}
    \|\nabla c_{\notparallel}(t)\|_{L^2}^2
    &\leq 2\left(\|\phi(t)\|_{L^2}^2+\|\psi(t)\|_{L^2}^2\right)\\\nonumber
    &\leq 2(M_0+400\|\partial_{x_1}c_{\notparallel}(0)\|_{L^2}^2)\\\nonumber
    &\leq 2(M_0+400):=Z_0,
\end{align}
for any $0<t\leq t_0$. 
It is important to note that the value of $Y_0$, $M_0$ and $Z_0$ are all independent of $\|\partial_{x_1}c_{\notparallel}(0)\|_{L^2}$, assuming that this quantity is sufficiently small.
\end{rem}

The rest of the paper is devoted to showing that under these assumptions, one can establish the same  estimates but with constants that are half the size, therefore if $t_0$ is defined to be the largest time on which the bootstrap assumptions hold, then $t_0$ can be extended by a fixed amount. Repeating this argument leads to $t_0=\infty$.
To this end, we observe  that each term appearing in  \eqref{eqn: energy_estimateI}, \eqref{eqn: energy_estimateII}, \eqref{eqn:phiL2}, and \eqref{eqn:psiL2esti} are well controlled provided $\|\partial_{x_2} c_{\parallel}\|^2_{L^2}$ is bounded. Therefore, we devote the next section to proving that this quantity is indeed uniformly bounded. 

\section{Uniform bounds for $\|\partial_{x_2} c_{\parallel}\|_{L^2}$} 
\label{sec:Boundness of u_par}

In this section, we  show that  $\left\|\partial_{x_2} c_{\parallel}\right\|_{L^2(x_2)}$ can be bounded from above uniformly for all $t \in [0, t_0]$, under the bootstrap assumptions and provided $\nu$ is sufficiently small. We recall that $c_{\parallel}$ satisfies \eqref{eqn: eqpar} and is a function of $x_2$ only. Next, we observe  that $\int_{\T}\partial_{x_2}^2 \left(3c_{\parallel}^2 c_{\notparallel}\right)dx_1=\partial_{x_2}^2 \left(3c_{\parallel}^2\int_{\T}c_{\notparallel}dx_1\right)= 0$, since $c_{\notparallel}$ has zero mean in the $x_1$ direction. Hence  \eqref{eqn: eqpar} can be written equivalently as 
\begin{equation}
\label{eqn: u_par}
    \partial_t c_{\parallel}=\nu \left(\partial_{x_2}^2  \left(c_{\parallel}^3-c_{\parallel}-\mu\partial_{x_2}^2 c_{\parallel}\right) +\int_{\T}\partial_{x_2}^2\left(3c_{\parallel}c_{\notparallel}^2+c_{\notparallel}^3\right)dx_1 \right).
\end{equation}

\begin{lem}
\label{lem:uniform_parallel}
Under the bootstrap assumptions \ref{i:bootstrap_a.1}--\ref{i:bootstrap_a.4}, 
there exists a constant $B_0>0$ depending on $\|\partial_{x_2}(c_0)_{\parallel}\|_{L^2}$, $M_0$, and $\|(c_0)_{\notparallel}\|_{L^2}$ such that 
\begin{equation}
\label{eqn:upbddofupar}
  \left\|\partial_{x_2} c_{\parallel} (t) \right\|_{L^2}^2\leq B_0,
\end{equation}
for any $t\in[0,t_0]$.
\end{lem}

\begin{rem}
\label{rem:energy_issue}
To carry out the bootstrap argument, a global control on $c_{\parallel}$ is needed. However, due to the negative sign in front of the term $\partial_{x_2}^2 c_{\parallel}$ on the right-hand side of \eqref{eqn: u_par}, it seems difficult to obtain a 
$t$-independent upper bound for $\int_0^{t}\|\partial_{x_2}^2 c_{\parallel}\|_{L^2}^2 ds$ via integration by parts and an $L^2$ estimate. Utilizing the energy structure of the Cahn-Hilliard equation, we can, however, derive a global-in-time bound on $\|\partial_{x_2} c_{\parallel}(t)\|_{L^2}$ by showing that the forcing term in \eqref{eqn: u_par}, which contains 
$c_{\notparallel}$, is integrable on the time interval over which the \emph{bootstrap estimates} hold. As discussed in the end of last section, this result in turn implies that the \emph{bootstrap estimates} can be extended on a larger time interval and hence globally. Therefore, we conclude that $c_{\parallel}$ is asymptotically close to a one-dimensional solution.
\end{rem}

\begin{proof}
For notational ease, we denote the  energy $E[c_{\parallel}(t)]$ by $E_{\parallel}(t)$, with $E$ defined in \eqref{eqn: energy functional}. Then multiplying  both sides of \eqref{eqn: u_par} by $c_{\parallel}^3-c_{\parallel}-\mu\partial_{x_2}^2 c_{\parallel}$  and integrating over $\TT^2$ yields:
\begin{align}
\label{eqn:esti_Epar}
    &\frac{d E_{\parallel}(t)}{dt}
    =-\nu\int_{\T}\left\vert\partial_{x_2}\left(c_{\parallel}^3-c_{\parallel}-\mu\partial_{x_2}^2 c_{\parallel}\right)\right\vert^2 dx_2\\\nonumber
    &\qquad \qquad +\nu\left<\partial_{x_2}^2\int_{\T}3c_{\parallel} c_{\notparallel}^2+c_{\notparallel}^3 dx_1, c_{\parallel}^3-c_{\parallel}-\mu\partial_{x_2}^2 c_{\parallel}\right>\\\nonumber
    \leq-\frac{\nu}{2} & \int_{\T}\left\vert\partial_{x_2}\left(c_{\parallel}^3-c_{\parallel}-\mu\partial_{x_2}^2 c_{\parallel}\right)\right\vert^2 dx_2 +C\nu\left\Vert\partial_{x_2}\left(3c_{\parallel} c_{\notparallel}^2+c_{\notparallel}^3 \right) \right\Vert_{L^2}^2\\\nonumber
    \leq C\nu\big( & \left\Vert c_{\notparallel}\,c_{\parallel}\,\partial_{x_2} c_{\notparallel}\right\Vert_{L^2}^2+\|\un^2\partial_{x_2}\up\|_{L^2}^2+\left\Vert c_{\notparallel}^2\,\partial_{x_2} c_{\notparallel}\right\Vert_{L^2}^2\big),
\end{align}
where $\langle,\rangle$ denotes the inner product in $L^2(\TT^2)$. 
By using \eqref{eq:2DGagliardo} again, the terms on the right-hand side can be controlled as follows:
\begin{align*}
    \left\Vert c_{\notparallel}\,c_{\parallel}\,\partial_{x_2} c_{\notparallel}\right\Vert_{L^2}^2
    &\leq\left\Vert c_{\notparallel}\right\Vert_{L^\infty}^2 \left\Vert c_{\parallel}\right\Vert_{L^\infty}^2\left\Vert \partial_{x_2}c_{\notparallel}\right\Vert_{L^2}^2 \\
    &\leq C \left\Vert\Delta c_{\notparallel}\right\Vert_{L^2}\left\Vert c_{\notparallel}\right\Vert_{L^2}\left\Vert \partial_{x_2}c_{\parallel}\right\Vert_{L^2}^2\left\Vert\nabla c_{\notparallel}\right\Vert_{L^2}^2,\\
    &\leq C E_{\parallel}(t)\left\Vert\Delta c_{\notparallel}\right\Vert_{L^2}^2\left\Vert c_{\notparallel}\right\Vert_{L^2}^2,\\
     \left\Vert c_{\notparallel}^2\,\partial_{x_2} c_{\parallel}\right\Vert_{L^2}^2
     &\leq \left\Vert c_{\notparallel}\right\Vert_{L^\infty}^{4}\left\Vert \partial_{x_2} c_{\parallel}\right\Vert_{L^2}^2\\
     &\leq C \left\Vert\Delta c_{\notparallel}\right\Vert_{L^2}^{2}\left\Vert c_{\notparallel}\right\Vert_{L^2}^{2}E_{\parallel}(t) ,\\
    \left\Vert c_{\notparallel}^2\,\partial_{x_2} c_{\notparallel}\right\Vert_{L^2}^2
    &\leq \left\Vert c_{\notparallel}\right\Vert_{L^\infty}^{4}\left\Vert \partial_{x_2} c_{\notparallel}\right\Vert_{L^2}^2\\
    &\leq C\left\Vert\Delta c_{\notparallel}\right\Vert_{L^2}^{2}\left\Vert c_{\notparallel}\right\Vert_{L^2}^{2}\left\Vert \psi\right\Vert_{L^2}^2,
\end{align*}
where we have used the fact that $\|\partial_{x_2} c_{\parallel}\|_{L^2}^2\leq C E_{\parallel}(t)$. Thus, the evolution of $E_{\parallel}(t)$ is governed by the following ODE:
\begin{align}
\label{eqn:E_par_fistest}
    \frac{d E_{\parallel}(t)}{dt}
    &\leq C\nu \left(\left\Vert\Delta c_{\notparallel}\right\Vert_{L^2}^2\left\Vert c_{\notparallel}\right\Vert_{L^2}^{2}\right) E_{\parallel}(t)\\\nonumber
    &\quad+C\nu \left(\left\Vert\Delta c_{\notparallel}\right\Vert_{L^2}^{2}\left\Vert c_{\notparallel}\right\Vert_{L^2}^{2}\left\Vert \psi\right\Vert_{L^2}^2\right)\\\nonumber
    &:=b(t)E_{\parallel}(t)+a(t).
\end{align}
Since $a(t)$ and $b(t)$ are both positive, solving the differential inequality \eqref{eqn:E_par_fistest} gives
\begin{align}
\label{eqn:E_par_B1}
E_{\parallel}(t)
&\leq\left(E_{\parallel}(0)+\int_0^t e^{-\int_0^s b(\tau)d\tau} a(s)ds\right)\exp{\left(\int_0^t b(s) ds\right)}\\\nonumber
&\leq\left(E_{\parallel}(0)+\int_0^t a(s)ds\right)\exp{\left(\int_0^t b(s) ds\right)}.
\end{align}
Then, by using the bootstrap assumptions \ref{i:bootstrap_a.1}--\ref{i:bootstrap_a.4}, for any $t\in(0,t_0]$ we have on one hand that
\begin{align}
\label{eqn: int_a}
\int_0^t a(s)ds
&=C\nu \int_0^t \left\Vert\Delta c_{\notparallel}\right\Vert_{L^2}^{2}\left\Vert c_{\notparallel}\right\Vert_{L^2}^{2}\left\Vert \psi\right\Vert_{L^2}^2 ds\\\nonumber
&\leq CM_0\|c_{\notparallel}(0)\|_{L^2}^{2}\left(\nu\int_0^t \|\Delta c_{\notparallel}\|_{L^2}^{2}ds\right)\\\nonumber
&\leq CM_0\|c_{\notparallel}(0)\|_{L^2}^{4},
\end{align}
while on the other hand,
\begin{align}
\label{eqn: int_b}
\int_0^t b(s) &ds
=C\nu \int_0^t  \left\Vert\Delta c_{\notparallel}\right\Vert_{L^2}^2\left\Vert c_{\notparallel}\right\Vert_{L^2}^{2} ds\\\nonumber
&\leq C\left\Vert c_{\notparallel}(0)\right\Vert_{L^2}^{4}.
\end{align}

Substituting \eqref{eqn: int_a} and \eqref{eqn: int_b} into \eqref{eqn:E_par_B1}, we finally obtain that
\begin{align*}
E_{\parallel}(t)
&\leq\left(E_{\parallel}(0)+CM_0\|c_{\notparallel}(0)\|_{L^2}^{4}\right)\exp{\left( C\left\Vert c_{\notparallel}(0)\right\Vert_{L^2}^{4}\right)}:=B_0.
\end{align*}
To complete the proof, we note that $E_{\parallel}(0)\leq C\Vert c_{\parallel}(0)\Vert_{L^4}^4+\Vert\partial_{x_2}c_{\parallel}(0)\Vert_{L^2}^2\leq C(\Vert\partial_{x_2}c_{\parallel}(0)\Vert_{L^2}^2+1)^2$ by \eqref{1dGNineqn}.
\end{proof}


\section{Bootstrap estimates} 
\label{sec:Bootstrap estimates}

In this section, we  carry out the bootstrap steps and prove Theorem \ref{thm:main}. To do so, we
show that, under the same assumptions as in Theorem \ref{thm:main}, if   \ref{i:bootstrap_a.1}--\ref{i:bootstrap_a.4} hold on $\left[0, t_0\right]$, then for any $0\leq s\leq t \leq t_0$ the following \emph{bootstrap estimates} hold
\begin{enumerate}[label={\bf (B\arabic*)}, ref=\textcolor{black}{\bf (B\arabic*)}] \label{l:bootstrap_e}
    \item  $\|c_{\notparallel}(t)\|_{L^2} \le 10e^{-\frac{\lambda_\nu(t-s)}{4}} \|c_{\notparallel}(s)\|_{L^2}$; 
    \label{i:bootstrap_e.1}
    
    \medskip
    
    \item $\mu\nu\int_{s}^{t}\|\Delta c_{\notparallel}(\tau)\|_{L^2}^2 d\tau \le 5\|c_{\notparallel}(s)\|_{L^2}^2$; 
    \label{i:bootstrap_e.2}

    \medskip

    \item  $\|\phi(t)\|_{L^2} \le 10e^{-\frac{\lambda_\nu(t-s)}{4}} \|\phi(s)\|_{L^2}$; 
    \label{i:bootstrap_e.3}

    \medskip

    \item  $\|\psi(t)\|_{L^2}^2\leq \widetilde{Y_0}\,\exp{\left[C_2\frac{5}{\mu}\Vert c_{\notparallel}(0)\Vert_{L^2}^4+1\right]}$,
    \label{i:bootstrap_e.4}
\end{enumerate}
where 
\begin{align}
\label{eqn:Y02}
    \widetilde{Y_0}&=2\Big(\|\psi(0)\|_{L^2}^2+C_1\left(\Vert c_{\notparallel}(0)\Vert_{L^2}^2+2\right)\Vert c_{\notparallel}(0)\Vert_{L^2}^2\\\nonumber
    &\qquad+\frac{\left(10 \Vert\partial_{x_2}v\Vert_{L^{\infty}}+1\right)^2}{4} \Big)<Y_0,
\end{align}
with $C_1$, $C_2$ the same constants as in \ref{i:bootstrap_a.4}. Therefore, we conclude that \ref{i:bootstrap_e.1}--\ref{i:bootstrap_e.4} hold for all positive times, from which Theorem \ref{thm:main} readily follows.

We start by establishing \ref{i:bootstrap_e.2}.

\begin{lem}
\label{lem:B2}
Assume the bootstrap assumptions \ref{i:bootstrap_a.1}--\ref{i:bootstrap_a.4}.  There exists $\nu_{01}$, which only depends on $\Vert\partial_{x_2}v\Vert_{L^{\infty}},$ $\Vert\partial_{x_2}(c_0)_{\parallel}\Vert_{L^2},$ $\Vert (c_0)_{\notparallel}\Vert_{L^2},$ $ \Vert\partial_{x_2}(c_0)_{\notparallel}\Vert_{L^2},$ 
such that 
\begin{equation}
    \mu\nu\int_{s}^{t}\|\Delta c_{\notparallel}(\tau)\|_{L^2}^2 d\tau \le 5\|c_{\notparallel}(s)\|_{L^2}^2,
\end{equation}
for any $0\leq s\leq t\leq t_0$ and for any $\nu\leq \nu_{01}$.
In particular, \ref{i:bootstrap_e.2} holds.
\end{lem}

\begin{proof}
We combine the bootstrap assumptions \ref{i:bootstrap_a.1}--\ref{i:bootstrap_a.4}  with  Lemma \ref{lem:uniform_parallel} and the energy identity \eqref{eqn: energy_estimateII} to obtain:
\begin{align*}
 &\mu\nu\int_s^t\|\Delta c_{\notparallel}(\tau)\|_{L^2}^2 d\tau
    \\\nonumber
    &\leq \| c_{\notparallel}(s)\|_{L^2}^2 
    + C\nu\Big(\int_s^t \| c_{\notparallel}\|_{L^2}^2 + \|\partial_{x_2} c_{\parallel}\|_{L^2}^4\|c_{\notparallel}\|_{L^2}^{2} + \|\partial_{x_2} c_{\parallel}\|_{L^2}^4 \|c_{\notparallel}\|_{L^2}^6\\\nonumber
    &\qquad\qquad\qquad\qquad\qquad+\|\nabla c_{\notparallel}\|_{L^2}^4\|c_{\notparallel}\|_{L^2}^2 d\tau\Big)\\
    &\leq \| c_{\notparallel}(s)\|_{L^2}^2 + C\frac{\nu}{\lambda_{\nu}}\Big(1+B_0^2+B_0^2\|c_{\notparallel}(0)\|_{L^2}^4+Z_0^2\Big)\| c_{\notparallel}(s)\|_{L^2}^2,
\end{align*}
where $B_0$ is the constant defined in Lemma \ref{lem:uniform_parallel} and $Z_0$ was defined in \eqref{eqn:Z0}. Moreover, since $\tfrac{\nu}{\lambda_{\nu}}\rightarrow 0$ as $\nu\rightarrow 0$, one can choose $\nu_{01}$ small enough such that
\begin{equation*}
    \frac{\nu_{01}}{\lambda_{\nu_{01}}}\leq\frac{4}{C\left(1+B_0^2+B_0^2\|c_{\notparallel}(0)\|_{L^2}^4+Z_0^2\right)},
\end{equation*}
which completes the proof.
\end{proof}

Before proceeding to establishing \ref{i:bootstrap_e.1}, we show that  $\|c_{\notparallel}\|_{L^2}$ can be controlled from above in a short time interval. We first observe that in view of assumption 
\ref{i:bootstrap_a.1}, it is only needed to control the $L^2$-norm on times of order $4/\lambda_\nu$ on the interval $[0,t_0]$.

\begin{lem}
\label{lem:L2control}
Assume the bootstrap assumptions \ref{i:bootstrap_a.1}--\ref{i:bootstrap_a.4} and let \mbox{$\tau^*=\frac{4}{\lambda_{\nu}}$}. For any $0\leq t_1\leq t_0$, there exists $\nu_{02}$, depending on $\Vert\partial_{x_2}v\Vert_{L^{\infty}},\Vert\partial_{x_2}(c_0)_{\parallel}\Vert_{L^2},$ $ \Vert (c_0)_{\notparallel}\Vert_{L^2}, \Vert\partial_{x_2}(c_0)_{\notparallel}\Vert_{L^2}$ such that for any $\nu\leq\nu_{02}$,
\begin{equation*}
    \|c_{\notparallel}(t)\|_{L^2}^2\leq 2\|c_{\notparallel}(t_1)\|_{L^2}^2
\end{equation*}
for any $t\in[t_1,t_1+\tau^*]\cap[0,t_0]$.
\end{lem}

\begin{proof}
By combining the bootstrap assumptions \ref{i:bootstrap_a.1}--\ref{i:bootstrap_a.4}, Lemma \ref{lem:uniform_parallel}, and the energy identity \eqref{eqn: energy_estimateI}, we obtain the following inequality on $[0,t_0]$:
\begin{align*}
    &\frac{d}{dt}\|c_{\notparallel}\|_{L^2}^2 \\
    &\leq 
    C\nu\left(\| c_{\notparallel}\|_{L^2}^2 + \|\partial_{x_2} c_{\parallel}\|_{L^2}^4\|c_{\notparallel}\|_{L^2}^{2} + \|\partial_{x_2} c_{\parallel}\|_{L^2}^4 \|c_{\notparallel}\|_{L^2}^6 + \|\nabla c_{\notparallel}\|_{L^2}^4\|c_{\notparallel}\|_{L^2}^2\right)\\
    &\leq C\nu\Big(\| c_{\notparallel}\|_{L^2}^2 + B_0^2 \|c_{\notparallel}\|_{L^2}^{2} + B_0^2\|c_{\notparallel}(0)\|_{L^2}^4\|c_{\notparallel}\|_{L^2}^2+Z_0^2\|c_{\notparallel}\|_{L^2}^2 \Big),
\end{align*}
with  $Z_0$ given  in \eqref{eqn:Z0}. Therefore, for $t\in [0,t_0]$
\begin{equation}
\label{eqn:ode_L2}
   \frac{d}{dt}\|c_{\notparallel}(t)\|_{L^2}^2 
    \leq \Tilde{C} \nu\|c_{\notparallel}(t)\|_{L^2}^2,
\end{equation}
where we have defined
\begin{equation*}
    \Tilde{C}=C(1+B_0^2+B_0^2\Vert c_{\notparallel}(0)\Vert^4_{L^2}+Z_0^2)
\end{equation*}
Consequently, for $t\in[t_1,t_1+\tau^*]\cap [0,t_0]$ we have
\begin{align*}
\Vert c_{\notparallel}(t)\Vert_{L^2}&\leq \Vert c_{\notparallel}(t_1)\Vert_{L^2}\exp(\tilde{C}\nu (t-t_1))\\
&\leq \Vert c_{\notparallel}(t_1)\Vert_{L^2}\exp(\tilde{C}\frac{4\nu}{\lambda_{\nu}}),
\end{align*}
using that   $\tau^*=\frac{4}{\lambda_{\nu}}$,
Then, again since  $\frac{\nu}{\lambda_{\nu}}\rightarrow 0$ as $\nu\rightarrow 0$, there exists $\nu_{02}$ depending only on 
$\Tilde{C}$, such that
\begin{equation*}
  \exp{\left(\tiny{\Tilde{C}}\frac{4\nu}{\lambda_{\nu}}\right)}\leq 2, 
\end{equation*}
for $\nu<\nu_{02}$.
In turn, this last estimate implies that
\begin{equation*}
  \|c_{\notparallel}(t)\|_{L^2}^2\leq\|c_{\notparallel}(t_1)\|_{L^2}^2\exp{\left(\tiny{\Tilde{C}}\frac{4\nu}{\lambda_{\nu}}\right)}
  \leq 2\|c_{\notparallel}(t_1)\|_{L^2}^2,
\end{equation*}
for any  $t\in[t_1,t_1+\tau^*]\cap[0,t_0]$. 
\end{proof}

Next, if $\tau^\ast$ is smaller than $t_0$,  the previous lemma is not sufficient to control  $\| c_{\notparallel}(t)\|_{L^2}$ on the interval $[0,t_0]$. In this case, however, we can use that the $L^2$ norm decays for times larger than $\tau^\ast$ by \ref{i:bootstrap_a.1}.

\begin{lem} \label{lem:preB1}
Assume the bootstrap assumptions \ref{i:bootstrap_a.1}--\ref{i:bootstrap_a.4},
and again let $\tau^*=4/\lambda_{\nu}$. There exists $\nu_{03}$, depending on $\Vert\partial_{x_2}v\Vert_{L^{\infty}},\Vert\partial_{x_2}(c_0)_{\parallel}\Vert_{L^2},$ $ \Vert (c_0)_{\notparallel}\Vert_{L^2}, \Vert\partial_{x_2}(c_0)_{\notparallel}\Vert_{L^2}$, such that, if $\tau^\ast<t_0$, then for any $0<\nu<\nu_{03}$ and for any $s\in[0,t_0-\tau^*]$, 
\begin{equation*}
   \| c_{\notparallel}(s+\tau^*)\|_{L^2}\leq \frac{1}{e}\| c_{\notparallel}(s)\|_{L^2}.
\end{equation*}
\end{lem}

\begin{proof}
For any $t\ge s$, we decompose $c_{\notparallel}(t)$ as 
\begin{equation}
\label{eqn:decomposition_u_notpar}
    c_{\notparallel}(t)=c_1(t)+c_2(t),
\end{equation}
where $c_1$ solves
\begin{equation*}
 \begin{cases} \partial_t c_1+v(x_2)\partial_{x_1} c_1+\mu\nu \Delta^2 c_1=0,\\
c_1(s)=c_{\notparallel}(s),
\end{cases}
\end{equation*}
and $c_2$ satisfies
\begin{equation}
\label{eqn:u2 equation}
 \begin{cases} 
  \partial_t c_2+v(x_2)\partial_{x_1} c_2+\mu\nu \Delta^2 c_2
   =-\nu\Delta c_{\notparallel} \\
    \quad +\nu\Delta\big(3 c_{\parallel}^2 c_{\notparallel} 
   +3c_{\parallel} c_{\notparallel}^2+c_{\notparallel}^3\big)-\nu\int_{\T}\partial_{x_2}^2\left(3c_{\parallel} c_{\notparallel}^2+c_{\notparallel}^3\right) dx_1,\\
   c_2(s)=0.
 \end{cases}
\end{equation}
Then, given the definition of $\tau^*$, the following bound holds for $c_1$:
\begin{equation}
\label{eq:c_1Bound}
   \|c_1(s+\tau^\ast)\|=\|e^{-\tau^\ast \,H_{\nu}}(c_{\notparallel}(s))\|_{L^2}\leq\frac{5}{e^4} \|c_{\notparallel}(s)\|_{L^2}\leq \frac{1}{e^2} \|c_{\notparallel}(s)\|_{L^2},
\end{equation}
since we defined $H_{\nu} = v(x_2)\partial_{x_1} +\mu\nu \Delta^2$.

Next, we perform a standard energy estimate on \eqref{eqn:u2 equation}:
 \begin{align}
    \partial_t\|c_2\|_{L^2}^2+2\mu\nu\|\Delta c_2\|_{L^2}^2
& =2\nu\left\langle \Delta  \left(-c_{\notparallel}+3 c_{\parallel}^2 c_{\notparallel}+3c_{\parallel} c_{\notparallel}^2+c_{\notparallel}^3\right), c_2\right\rangle \nonumber \\
 &\qquad -2\nu\left\langle \int_{\T}\partial_{x_2}^2\left(3c_{\parallel} c_{\notparallel}^2+c_{\notparallel}^3\right) dx_1, c_2\right\rangle, 
  \label{eqn:energy u2}
 \end{align}
 so that
\begin{align} \nonumber
&\partial_t\|c_2\|_{L^2}^2+2\mu\nu\|\Delta c_2\|_{L^2}^2 \leq C\nu \Big(\|c_{\notparallel}\|_{L^2}\|\Delta c_2\|_{L^2}+\| c_{\parallel}^2 c_{\notparallel}\|_{L^2}\|\Delta c_2\|_{L^2}+\\\nonumber
&\qquad\qquad+\| c_{\parallel} c_{\notparallel}^2\|_{L^2}\|\Delta c_2\|_{L^2}+\|c_{\notparallel}^3\|_{L^2}\|\Delta c_2\|_{L^2}\Big)\\\nonumber
&\leq \mu\nu\|\Delta c_2\|_{L^2}^2 +C\nu\left(\|c_{\notparallel}\|_{L^2}^2+\| c_{\parallel}^2 c_{\notparallel}\|_{L^2}^2+\| c_{\parallel} c_{\notparallel}^2\|_{L^2}^2 +\|c_{\notparallel}^3\|_{L^2}^2\right)\\\nonumber
&\leq \mu\nu\|\Delta c_2\|_{L^2}^2 +C\nu\left(\|c_{\notparallel}\|_{L^2}^2+\| c_{\parallel}\|^4_{L^{\infty}} \|c_{\notparallel}\|_{L^2}^2+\| c_{\parallel}\|_{L^{\infty}}^2 \|c_{\notparallel}\|_{L^4}^4 +\|c_{\notparallel}\|_{L^6}^6\right) \\ \nonumber
&\leq \mu\nu\|\Delta c_2\|_{L^2}^2 +C\nu\Big(\|c_{\notparallel}\|_{L^2}^2+\|\partial_{x_2} c_{\parallel}\|_{L^2}^4\|c_{\notparallel}\|_{L^2}^{2}\\ \nonumber
&\qquad\qquad\qquad\qquad\qquad+\|\partial_{x_2} c_{\parallel}\|_{L^2}^2\|\nabla c_{\notparallel}\|_{L^2}^{2}\|c_{\notparallel}\|_{L^2}^{2}+\|\nabla c_{\notparallel}\|_{L^2}^4\|c_{\notparallel}\|_{L^2}^2\Big)
\end{align}
where we used the Gagliardo–Nirenberg inequalities in \eqref{1dGNineqn} and \eqref{eqn:three_GN_ineq}. Integrating the above inequality over the interval $[s,s+\tau^\ast]$ gives:
\begin{align}
\label{eq:c_2Bound}
 &\|c_2(s+\tau^{*})\|_{L^2}^2\leq C\nu\Big(\int_s^{s+\tau^{*}} \|c_{\notparallel}\|_{L^2}^2+\|\partial_{x_2} c_{\parallel}\|_{L^2}^4\|c_{\notparallel}\|_{L^2}^{2} 
\\\nonumber
 &+\|\partial_{x_2} c_{\parallel}\|_{L^2}^2\|\nabla c_{\notparallel}\|_{L^2}^{2}\|c_{\notparallel}\|_{L^2}^{2}
+\|\nabla c_{\notparallel}\|_{L^2}^4\|c_{\notparallel}\|_{L^2}^2 d\tau\Big).
\end{align}
Hence,  \eqref{eq:c_1Bound} and \eqref{eq:c_2Bound} implies that
\begin{align}
&\|c_{\notparallel}(s+\tau^*)\|_{L^2} \leq \|e^{-\tau^\ast \,H_{\nu}}(c_{\notparallel}(s))\|_{L^2} \\\nonumber
&\quad+\Big(C\nu\int_s^{s+\tau^*} \|\partial_{x_2} c_{\parallel}\|_{L^2}^4\|c_{\notparallel}\|_{L^2}^{2}+\|\partial_{x_2} c_{\parallel}\|_{L^2}^2\|\nabla c_{\notparallel}\|_{L^2}^{2}\|c_{\notparallel}\|_{L^2}^{2}\\\nonumber
&\qquad\qquad\qquad\quad+\|\nabla c_{\notparallel}\|_{L^2}^4\|c_{\notparallel}\|_{L^2}^2+\|c_{\notparallel}\|_{L^2}^2 d\tau\Big)^{1/2}\\\nonumber
&\leq \frac{1}{e^2} \|c_{\notparallel}(s)\|_{L^2}+C\left(\frac{\nu}{\lambda_{\nu}}\right)^{1/2}\| c_{\notparallel}(s)\|_{L^2}\,\left(B_0+Z_0+1\right),
\end{align}
where we have used the bootstrap assumptions \ref{i:bootstrap_a.1}--\ref{i:bootstrap_a.4}  and Lemma \ref{lem:uniform_parallel}. Finally, exploiting again that $\nu/\lambda_\nu\to 0$ as $\nu\to 0$,  there exists $\nu_{03}$ small enough such that  for $\nu< \nu_{03}$,
\begin{align*}
    C\left(\frac{\nu}{\lambda_{\nu}}\right)^{1/2}<\frac{e^{-1}-e^{-2}}{B_0+Z_0+1}.
\end{align*}
 which  yields the desired inequality.
\end{proof}

\begin{rem}
Decomposing $c_{\notparallel}$ as in \eqref{eqn:decomposition_u_notpar} allows to better bound $\|c_{\notparallel}(s+\tau^*)\|_{L^2}$  than directly applying Duhamel's principle, for instance, as in \cite{coti2021global, FHX21B, FFIT19, FM22, iyer2021convection}. In particular, when the nonlinear term is in divergence form, the decomposition \eqref{eqn:decomposition_u_notpar} allows  to perform an integration by parts to reduce the order of derivatives on the right-hand side  as in \eqref{eqn:energy u2} . By comparison, in \cite{FHX21B} the authors considered an advective thin-film equation with forcing in  divergence form. By utilizing Duhamel's form, they proved that if the dissipation time  $\tau^{*}(v)$ (see Definition 3.6 in \cite{FHX21B}) satisfies a certain condition (Condition (3.9) in \cite{FHX21B}), then the flow can suppress finite-time blow up giving a global weak solution. If instead one decomposes the solution as in \eqref{eqn:decomposition_u_notpar} to redo the estimate of Lemma 3.14 in \cite{FHX21B},  Condition (3.9) can be substituted for by the weaker and more natural assumption $\tau^{*}(v)<\tilde{B}$, where $\tilde{B}>0$ is a threshold value. We should stress, as a matter of fact, that it is unclear whether an explicit time-independent flow exists that satisfies condition (3.9) in \cite{FHX21B}, while there exist flows satisfying the weaker condition, such as  certain cellular flows with  an appropriate amplitude and cell size so that $\tau^{*}(v)<\tilde{B}$ (see \cite{F2022dissipation, iyer2021convection, iyer2022quantifying}).
\end{rem}

Combining Lemma \ref{lem:L2control} and Lemma \ref{lem:preB1} allows finally to establish the bootstrap estimate \ref{i:bootstrap_e.1}.

\begin{lem}
\label{lem:B1}
Assume again the bootstrap assumptions \ref{i:bootstrap_a.1}--\ref{i:bootstrap_a.4}. Let $\widetilde{\nu_0}=\min\{\nu_{02}, \nu_{03}\}$, where $\nu_{02}$ and $\nu_{03}$ are as in Lemma \ref{lem:L2control} and  \ref{lem:preB1}, respectively. 
Then for any $0\leq s\leq t\leq t_0$ and  $\nu\leq\widetilde{\nu_0}$ ,
\begin{equation}
\|c_{\notparallel}(t)\|_{L^2} \le 10\, e^{-\frac{\lambda_\nu(t-s)}{4}} \,\|c_{\notparallel}(s)\|_{L^2}.
\end{equation}
In particular, \ref{i:bootstrap_e.1} holds.
\end{lem}

\begin{proof}
Since $\widetilde{\nu_0}=\min\{\nu_{02}, \nu_{03}\}$, Lemma \ref{lem:L2control} and Lemma \ref{lem:preB1} hold. If $t_0<\tau^{*}$, \ref{i:bootstrap_e.1}  follows directly from Lemma \ref{lem:L2control} since 
\begin{equation*}
\|c_{\notparallel}(t)\|_{L^2}\leq \sqrt{2}\|c_{\notparallel}(s)\|_{L^2}\leq \frac{10}{e}\|c_{\notparallel}(s)\|_{L^2}\leq 10\, e^{-\frac{\lambda_\nu(t-s)}{4}} \,\|c_{\notparallel}(s)\|_{L^2}.
\end{equation*}
If $t_0\geq \tau^{*}$, Lemma \ref{lem:preB1} gives that for any $n\in\mathbb{N}$ with $s+n\tau^{*}\leq t_0$, 
\begin{equation*}
\|c_{\notparallel}(s+n\tau^{*})\|_{L^2}
\leq e^{-n} \|c_{\notparallel}(s)\|_{L^2}.
\end{equation*}
But, for any $0\leq s\leq t\leq t_0$, there exists $n\in \NN$ with the property that $t$ belongs to the interval $[s+n\tau^{*}, s+(n + 1)\tau^{*})$. Then Lemma \ref{lem:L2control} with $t_1 =s+n\tau^{*}$ gives that
\begin{align*}
    \|c_{\notparallel}(t)\|_{L^2}
    &\leq \sqrt{2}\|c_{\notparallel}(s+n\tau^{*})\|_{L^2}
    \leq \sqrt{2}e^{-n}\|c_{\notparallel}(s)\|_{L^2}\\
    &\leq \sqrt{2}e^{1-(t-s)/\tau^{*}}\|c_{\notparallel}(s)\|_{L^2}
    \leq 10 \, e^{-\frac{\lambda_\nu(t-s)}{4}}\, \|c_{\notparallel}(s)\|_{L^2},
\end{align*}
as  $t<s+(n+1)\tau^{*}$, hence $-n\leq 1-\frac{t-s}{\tau^*}$.
\end{proof}

Next, we turn to the proof of \ref{i:bootstrap_e.3}, which is similar to the proof of \ref{i:bootstrap_e.1}. We first prove the following lemma, corresponding to Lemma \ref{lem:L2control} for \ref{i:bootstrap_e.1}. We recall that  $\phi$ denotes $\partial_{x_1}c_{\notparallel}$ and satisfies equation \eqref{eqn:phieqn}.

\begin{lem}
\label{lem:phiL2control}
Assume the bootstrap assumptions \ref{i:bootstrap_a.1}--\ref{i:bootstrap_a.4}, and let $\tau^*=\frac{4}{\lambda_{\nu}}$. For any $0\leq t_1\leq t_0$, there exists $\nu_{04}$, depending on $\Vert\partial_{x_2}v\Vert_{L^{\infty}},\Vert\partial_{x_2}(c_0)_{\parallel}\Vert_{L^2},$ $ \Vert (c_0)_{\notparallel}\Vert_{L^2}, \Vert\partial_{x_2}(c_0)_{\notparallel}\Vert_{L^2}$, such that for any $\nu\leq\nu_{04}$ one has
\begin{equation*}
    \|\phi(t)\|_{L^2}^2\leq 2\|\phi(t_1)\|_{L^2}^2
\end{equation*}
for any $t\in[t_1,t_1+\tau^*]\cap[0,t_0]$. 
\end{lem}

\begin{proof}
By combining the bootstrap assumptions \ref{i:bootstrap_a.1}--\ref{i:bootstrap_a.4}, Lemma \ref{lem:uniform_parallel}, Remark \ref{rmk1}, and the energy inequality \eqref{eqn:phiL2}, we obtain:
\begin{align*}
\frac{d}{dt}\|\phi\|_{L^2}^2
&\leq C\nu\Big(1+\Vert \partial_{x_2} c_{\parallel}\Vert_{L^2}^4+\Vert \partial_{x_2} c_{\parallel}\Vert_{L^2}^2\Vert\Delta c_{\notparallel}\Vert_{L^2}\Vert c_{\notparallel}\Vert_{L^2}\\\nonumber
&\qquad\qquad+\Vert \Delta c_{\notparallel}\Vert_{L^2}^{4/3}\Vert \nabla c_{\notparallel}\Vert_{L^2}^{4/3}\Vert c_{\notparallel}\Vert_{L^2}^{4/3}\Big)\Vert\phi\Vert_{L^2}^2\\
&\leq C\nu\Big(1+B_0^2+B_0\Vert\Delta c_{\notparallel}\Vert_{L^2}\Vert c_{\notparallel}\Vert_{L^2}+Z_0^{2/3}\Vert\Delta c_{\notparallel}\Vert_{L^2}^{4/3}\Vert c_{\notparallel}\Vert_{L^2}^{4/3}\Big)\Vert\phi\Vert_{L^2}^2.
\end{align*}
Therefore, there exists a constant $\hat{C}$, depending only on $B_0$ and $Z_0$, such that
for $t\in [0,t_0]$,
\begin{equation}
   \frac{d}{dt}\|\phi(t)\|_{L^2}^2 
    \leq \hat{C} \nu\left(1+\Vert\Delta c_{\notparallel}\Vert_{L^2}\Vert c_{\notparallel}\Vert_{L^2}+\Vert\Delta c_{\notparallel}\Vert_{L^2}^{4/3}\Vert c_{\notparallel}\Vert_{L^2}^{4/3}\right)\Vert\phi\Vert_{L^2}^2.
\end{equation}
Consequently, for $t\in[t_1,t_1+\tau^*]\cap [0,t_0]$, where  again
$\tau^*=\frac{4}{\lambda_{\nu}}$, one has
\begin{align}
&\ln{\left(\|\phi(t)\|_{L^2}^2\right)} \leq\ln{\left(\|\phi(t_1)\|_{L^2}^2\right)}  +  
 \nonumber \\
 & \qquad \qquad \qquad \hat{C}\nu\int_{t_1}^{t}\left(1+\Vert\Delta c_{\notparallel}\Vert_{L^2}\Vert c_{\notparallel}\Vert_{L^2}+\Vert\Delta c_{\notparallel}\Vert_{L^2}^{4/3}\Vert c_{\notparallel}\Vert_{L^2}^{4/3}\right)ds \nonumber \\
&\leq\ln{\left(\|\phi(t_1)\|_{L^2}^2\right)}+\hat{C}\frac{\nu}{\lambda_{\nu}}+\hat{C}\left(\nu\int_{0}^{t}\Vert\Delta c_{\notparallel}\Vert_{L^2}^2 ds\right)^{1/2}\left(\nu\int_{0}^{t}\Vert c_{\notparallel}\Vert_{L^2}^2 ds\right)^{1/2}
\nonumber \\
&\qquad+\hat{C}\left(\nu\int_{0}^{t}\Vert\Delta c_{\notparallel}\Vert_{L^2}^2 ds\right)^{2/3}\left(\nu\int_{0}^{t}\Vert c_{\notparallel}\Vert_{L^2}^4 ds\right)^{1/3}
\nonumber \\
&\leq\ln{\left(\|\phi(t_1)\|_{L^2}^2\right)}+\hat{C}\left(\frac{\nu}{\lambda_{\nu}}+\left(\frac{\nu}{\lambda_{\nu}}\right)^{1/2}\Vert c_{\notparallel}(0)\Vert_{L^2}^2+\left(\frac{\nu}{\lambda_{\nu}}\right)^{1/3}\Vert c_{\notparallel}(0)\Vert_{L^2}^{8/3}\right), \nonumber
\end{align}
where we used \ref{i:bootstrap_a.1} and \ref{i:bootstrap_a.2} to establish  the last inequality. Above, $\hat{C}$ may change from line to line by an absolute constant. Then, again since $\frac{\nu}{\lambda_{\nu}}\rightarrow 0$ as $\nu\rightarrow 0$, there exists $\nu_{04}$, depending on $B_0$, $Z_0$, $\Vert (c_0)_{\notparallel}\Vert_{L^2}$, hence depending only on $\Vert\partial_{x_2}v\Vert_{L^{\infty}},\Vert\partial_{x_2}(c_0)_{\parallel}\Vert_{L^2},$ $ \Vert (c_0)_{\notparallel}\Vert_{L^2}, \Vert\partial_{x_2}(c_0)_{\notparallel}\Vert_{L^2}$, such that
\begin{equation*}
  \hat{C}\left(\frac{\nu}{\lambda_{\nu}}+\left(\frac{\nu}{\lambda_{\nu}}\right)^{1/2}\Vert c_{\notparallel}(0)\Vert_{L^2}^2+\left(\frac{\nu}{\lambda_{\nu}}\right)^{1/3}\Vert c_{\notparallel}(0)\Vert_{L^2}^{8/3}\right)\leq \ln{2}
\end{equation*}
for $\nu<\nu_{04}$. This estimate implies that
\begin{equation*}
  \|\phi(t)\|_{L^2}^2
  \leq 2\|\phi(t_1)\|_{L^2}^2,
\end{equation*}
for any  $t\in[t_1,t_1+\tau^*]\cap[0,t_0]$.
\end{proof}

Next,  similarly to Lemma \ref{lem:preB1}, we have the following estimate for $\phi$.

\begin{lem}
\label{lem:phipreB1}
Assume the bootstrap assumptions \ref{i:bootstrap_a.1}--\ref{i:bootstrap_a.4},
and let again $\tau^*=4/\lambda_{\nu}$. There exists $\nu_{05}$, depending on $\Vert\partial_{x_2}v\Vert_{L^{\infty}}$, $\Vert\partial_{x_2}(c_0)_{\parallel}\Vert_{L^2}$, $ \Vert (c_0)_{\notparallel}\Vert_{L^2}$, $\Vert\partial_{x_2}(c_0)_{\notparallel}\Vert_{L^2}$, such that for any $0<\nu<\nu_{05}$, if $\tau^\ast<t_0$, then for any $s\in[0,t_0-\tau^*]$, 
\begin{equation*}
   \| \phi(s+\tau^*)\|_{L^2}\leq \frac{1}{e}\| \phi(s)\|_{L^2}.
\end{equation*}
\end{lem}

\begin{proof}
For any $t\ge s$, We decompose $\phi$ in the same manner as in \eqref{eqn:decomposition_u_notpar}:
\begin{equation}
 \phi(t)=\phi_1(t)+\phi_2(t)\,,
\end{equation}
with $\phi_1$ satisfying
\begin{equation*}
 \begin{cases} \partial_t \phi_1+v(x_2)\partial_{x_1} \phi_1+\mu\nu \Delta^2 \phi_1=0,\\
\phi_1(s)=\phi(s),
\end{cases}
\end{equation*}
and $\phi_2$ satisfying
\begin{equation*}
 \begin{cases} 
     \partial_t\phi_2+v(x_2)\partial_{x_1}\phi_2+\mu\nu\Delta^2\phi_2
    =-\nu\Delta\phi+\nu\Delta\left(\partial_{x_1}\left(3 c_{\parallel}^2 c_{\notparallel}+3 c_{\parallel}c_{\notparallel}^2+c_{\notparallel}^3\right)\right),\\
   \phi_2(s)=0.
 \end{cases}
\end{equation*}
Then,  recalling the definition of $\tau^*$, the following bound holds for $\phi_1$:
\begin{equation}
\label{eq:phi_1Bound}
   \|\phi_1(s+\tau^\ast)\|=\|e^{-\tau^\ast \,H_{\nu}}(\phi(s))\|_{L^2}\leq\frac{1}{e^2} \|\phi(s)\|_{L^2}.
\end{equation}
Next, it is straightforward to see that $\phi_2$ satisfies the same estimate as in \eqref{eqn:phiL2}. Integrating this estimate over the time interval $[s,s+\tau^\ast]$ gives:
\begin{align}
 &\|\phi_2(s+\tau^{*})\|_{L^2}^2\leq C\nu\int_s^{s+\tau^{*}} \Big(1+\Vert \partial_{x_2} c_{\parallel}\Vert_{L^2}^4+\Vert \partial_{x_2} c_{\parallel}\Vert_{L^2}^2\Vert\Delta c_{\notparallel}\Vert_{L^2}\Vert c_{\notparallel}\Vert_{L^2} \nonumber
\\
 &\qquad\qquad\qquad\qquad+\Vert \Delta c_{\notparallel}\Vert_{L^2}^{4/3}\Vert\nabla c_{\notparallel}\Vert_{L^2}^{4/3}\Vert c_{\notparallel}\Vert_{L^2}^{4/3}\Big)\Vert\phi\Vert_{L^2}^2 d\tau.
 \label{eq:phi_2Bound}
\end{align}
Combining \eqref{eq:phi_1Bound} and \eqref{eq:phi_2Bound}  implies that
\begin{align}
&\|\phi(s+\tau^*)\|_{L^2} 
\leq \|e^{-\tau^\ast \,H_{\nu}}(\phi(s))\|_{L^2} \\\nonumber
&\quad+\Big(C\nu\int_s^{s+\tau^{*}} \Big(1+\Vert \partial_{x_2} c_{\parallel}\Vert_{L^2}^4+\Vert \partial_{x_2} c_{\parallel}\Vert_{L^2}^2\Vert\Delta c_{\notparallel}\Vert_{L^2}\Vert c_{\notparallel}\Vert_{L^2}
\\\nonumber
 &\qquad\qquad\qquad\qquad+\Vert \Delta c_{\notparallel}\Vert_{L^2}^{4/3}\Vert\nabla c_{\notparallel}\Vert_{L^2}^{4/3}\Vert c_{\notparallel}\Vert_{L^2}^{4/3}\Big) \norm{\phi}_{L^2}d\tau\Big)^{1/2}\\\nonumber
&\leq \frac{1}{e^2} \|\phi(s)\|_{L^2}+C\Vert\phi(s)\Vert_{L^2}\Big(\nu\int_s^{s+\tau^{*}} \Big(1+B_0^2+B_0\Vert\Delta c_{\notparallel}\Vert_{L^2}\Vert c_{\notparallel}\Vert_{L^2}\\\nonumber
&\quad\quad\qquad\qquad\qquad\qquad\qquad+Z_0^{2/3}\Vert\Delta c_{\notparallel}\Vert_{L^2}^{4/3}\Vert c_{\notparallel}\Vert_{L^2}^{4/3}\Big)d\tau\Big)^{1/2}\\\nonumber
\end{align}
where we have used Lemma \ref{lem:phiL2control} in the last inequality. 
From the bootstrap assumptions \ref{i:bootstrap_a.1}--\ref{i:bootstrap_a.4}, Lemma \ref{lem:uniform_parallel}, and Lemma \ref{lem:phiL2control} it follows that
\begin{align}
&\|\phi(s+\tau^*)\|_{L^2} \leq \frac{1}{e^2} \|\phi(s)\|_{L^2}+C(1+B_0)\left(\frac{\nu}{\lambda_{\nu}}\right)^{1/2}\|\phi(s)\|_{L^2}\\\nonumber
&\quad+C B_0^{1/2}\left(\nu\int_0^{s+\tau^{*}}\Vert\Delta c_{\notparallel}\Vert_{L^2}\Vert c_{\notparallel}\Vert_{L^2}d\tau\right)^{1/2}\|\phi(s)\|_{L^2}\\\nonumber
&\quad+C Z_0^{1/3}\left(\nu\int_0^{s+\tau^{*}}\Vert\Delta c_{\notparallel}\Vert_{L^2}^{4/3}\Vert c_{\notparallel}\Vert_{L^2}^{4/3}d\tau\right)^{1/2}\|\phi(s)\|_{L^2}\\\nonumber
&\leq \frac{1}{e^2} \|\phi(s)\|_{L^2}+C\Big((1+B_0)\left(\frac{\nu}{\lambda_{\nu}}\right)^{1/2}+ B_0^{1/2}\Vert c_{\notparallel}(0)\Vert_{L^2}\left(\frac{\nu}{\lambda_{\nu}}\right)^{1/4}\\\nonumber
&\qquad\qquad\qquad\qquad\qquad+Z_0^{1/3}\Vert c_{\notparallel}(0)\Vert_{L^2}^{4/3}\left(\frac{\nu}{\lambda_{\nu}}\right)^{1/6}\Big)\|\phi(s)\|_{L^2},
\end{align}
Finally, utilizing again the fact that $\nu/\lambda_\nu\to 0$ as $\nu\to 0$, we conclude that there exists $\nu_{05}$, depending on $\Vert\partial_{x_2}v\Vert_{L^{\infty}}$, $\Vert\partial_{x_2}(c_0)_{\parallel}\Vert_{L^2}$, $ \Vert (c_0)_{\notparallel}\Vert_{L^2}$, $\Vert\partial_{x_2}(c_0)_{\notparallel}\Vert_{L^2}$, such that 
\begin{align*}
&C\Big((1+B_0)\left(\frac{\nu}{\lambda_{\nu}}\right)^{1/2}+ B_0^{1/2}\Vert c_{\notparallel}(0)\Vert_{L^2}\left(\frac{\nu}{\lambda_{\nu}}\right)^{1/4}\\\nonumber
&\qquad\qquad\qquad\qquad\qquad\qquad\qquad+Z_0^{1/3}\Vert c_{\notparallel}(0)\Vert_{L^2}^{4/3}\left(\frac{\nu}{\lambda_{\nu}}\right)^{1/6}\Big)<\frac{1}{e}-\frac{1}{e^2}
\end{align*}
for $\nu<\nu_{05}$.
\end{proof}

By using similar arguments as in the proof of Lemma \ref{lem:B1}, we can establish the validity of \ref{i:bootstrap_e.3} from Lemma \ref{lem:phiL2control} and Lemma \ref{lem:phipreB1}. We therefore omit the proof of the next lemma for brevity.

\begin{lem}
\label{lem:B3}
Assume again the bootstrap assumptions \ref{i:bootstrap_a.1}--\ref{i:bootstrap_a.4}. Let $\widehat{\nu_0}=\min\{\nu_{04},\nu_{05}\}$, where $\nu_{04}$ and $\nu_{05}$ are as in Lemma \ref{lem:phiL2control} and \ref{lem:phipreB1}, respectively.
Then for any $0\leq s\leq t\leq t_0$ and for any $\nu\leq\widehat{\nu_0}$,
\begin{equation}
\|\phi(t)\|_{L^2} \le 10\, e^{-\frac{\lambda_\nu(t-s)}{4}} \,\|\phi(s)\|_{L^2}.
\end{equation}
In particular, \ref{i:bootstrap_e.3} holds.
\end{lem}

Before we tackle \ref{i:bootstrap_e.4}, it is important to observe that in view of Remark \ref{rmk1}, Lemma \ref{lem:uniform_parallel},  Lemma \ref{lem:B2}, Lemma \ref{lem:B3}, the constants  $\widetilde{\nu_0}$ and $\widehat{\nu_0}$ introduced above are independent of $\Vert\partial_{x_1}\left(c_0\right)_{\notparallel}\Vert_{L^2}$. We are now ready to prove \ref{i:bootstrap_e.4} holds assuming the smallness condition \eqref{eqn:smallinitial}.

\begin{lem}
\label{lem:B4}
Assume again the bootstrap assumptions \ref{i:bootstrap_a.1}--\ref{i:bootstrap_a.4}. There exists $0<\nu_0<1$, depending on $\|\partial_{x_2}v\|_{L^\infty}$, $\|(c_0)_{\parallel}\|_{L^2}$, $\|\partial_{x_2}(c_0)_{\parallel}\|_{L^2}$, $\|(c_0)_{\notparallel}\|_{L^2}$, $\|\partial_{x_2}(c_0)_{\notparallel}\|_{L^2}$, such that, if for any $0<\nu\leq\nu_0$ the initial data $c_0$  satisfies
\begin{equation}
\label{eqn:smallinitial2}
    \|\partial_{x_1}\left(c_0\right)_{\notparallel}\|_{L^2}\leq \min{\left\{\lambda_{\nu},1\right\}},
\end{equation}
 then
\begin{equation}
\|\psi(t)\|_{L^2}^2\leq \widetilde{Y_0}\,\exp{\left[C_2\frac{5}{\mu}\Vert c_{\notparallel}(0)\Vert_{L^2}^4+1\right]}
\end{equation}
for any $t>0$, with $\widetilde{Y_0}$ defined in \eqref{eqn:Y02}. In particular, \ref{i:bootstrap_e.4} holds.
\end{lem}

\begin{proof}
Integrating in time estimate\eqref{eqn:psiL2esti}  for $\|\psi\|_{L^2}^2$) yields:
\begin{align*}
&\|\psi(t)\|_{L^2}^2 \leq \|\psi(0)\|_{L^2}^2\\
&\quad+C^\ast\nu\int_0^t \left(\Vert\Delta c_{\notparallel}\Vert_{L^2}^2\Vert c_{\notparallel}\Vert_{L^2}^2+\Vert\partial_{x_2}c_{\parallel}\Vert_{L^2}^2\Vert\Delta c_{\notparallel}\Vert_{L^2}\Vert c_{\notparallel}\Vert_{L^2}\right)\Vert\psi\Vert_{L^2}^2ds\\
&\quad+\int_0^t\left(\Vert\partial_{x_2}v\Vert_{L^{\infty}}\Vert\phi\Vert_{L^2}+C^\ast\nu\Vert\partial_{x_2} c_{\parallel}\Vert_{L^2}^3\Vert\Delta c_{\notparallel}\Vert_{L^2}\Vert c_{\notparallel}\Vert_{L^2}\right)\Vert\psi\Vert_{L^2}ds\\
&\quad +C^\ast\nu\Big(\int_0^t\Vert \Delta c_{\notparallel}\Vert_{L^2}^2+\Vert\partial_{x_2} c_{\parallel}\Vert_{L^2}^4\Vert \Delta c_{\notparallel}\Vert_{L^2}\Vert c_{\notparallel}\Vert_{L^2}\\
&\qquad\qquad+\Vert\partial_{x_2} c_{\parallel}\Vert_{L^2}^3\Vert\Delta c_{\notparallel}\Vert_{L^2}\Vert c_{\notparallel}\Vert_{L^2}\Vert\phi\Vert_{L^2}+\Vert\Delta c_{\notparallel}\Vert_{L^2}^2\Vert c_{\notparallel}\Vert_{L^2}^2\Vert\phi\Vert_{L^2}^2ds\Big).
\end{align*}
To bound the last term on the right-hand side, we exploit the bootstrap assumptions  \ref{i:bootstrap_a.1}--\ref{i:bootstrap_a.3}, the  bound $\|\partial_{x_1}\left(c_0\right)_{\notparallel}\|_{L^2}\leq 1$,  and the fact we can choose $\nu_{06}$ small enough so  that for any $\nu<\nu_{06}$ we have
\begin{align}
\label{eqn:smallness01}
\left( B_0^2+B_0^{3/2}\right)\left(\frac{\nu}{\lambda_{\nu}}\right)^{1/2}\leq 1.
\end{align}
Then we have that

\begin{align}
&C^\ast\nu\Big(\int_0^t\Vert \Delta c_{\notparallel}\Vert_{L^2}^2+\Vert\partial_{x_2} c_{\parallel}\Vert_{L^2}^4\Vert \Delta c_{\notparallel}\Vert_{L^2}\Vert c_{\notparallel}\Vert_{L^2} \nonumber\\
\nonumber
&\qquad\qquad+\Vert\partial_{x_2} c_{\parallel}\Vert_{L^2}^3\Vert\Delta c_{\notparallel}\Vert_{L^2}\Vert c_{\notparallel}\Vert_{L^2}\Vert\phi\Vert_{L^2}+\Vert\Delta c_{\notparallel}\Vert_{L^2}^2\Vert c_{\notparallel}\Vert_{L^2}^2\Vert\phi\Vert_{L^2}^2ds\Big)\\\nonumber
&= C^\ast\nu \Big(\int_0^t (1+\Vert c_{\notparallel}\Vert_{L^2}^2\Vert \norm{\phi}_{L^2}^2) \Vert\Delta c_{\notparallel}\Vert_{L^2}^2\\\nonumber
&\qquad\qquad+(\Vert\partial_{x_2} c_{\parallel}\Vert_{L^2}^4 +\Vert\partial_{x_2} c_{\parallel}\Vert_{L^2}^3\Vert\phi\Vert_{L^2})\Vert \Delta c_{\notparallel}\Vert_{L^2}\Vert c_{\notparallel}\Vert_{L^2}ds\Big)\\\nonumber
&\leq C^\ast\left((20\Vert c_{\notparallel}(0)\Vert_{L^2})^2(20\Vert \partial_{x_1}c_{\notparallel}(0)\Vert_{L^2})^2+1\right)\left(\nu\int_0^t\Vert\Delta c_{\notparallel}\Vert_{L^2}^2ds\right) + \\\nonumber
&C^\ast \sqrt{\frac{\nu}{\lambda_{\nu}}} \left( B_0^2+20B_0^{3/2}\Vert \partial_{x_1}c_{\notparallel}(0)\Vert_{L^2}\right)\left(\nu\int_0^t\Vert\Delta c_{\notparallel}\Vert_{L^2}^2ds\right)^{1/2}
(20\Vert c_{\notparallel}(0)\Vert_{L^2})\\\nonumber
&\leq C_1\left[\left(\Vert c_{\notparallel}(0)\Vert_{L^2}^2+1\right)\Vert c_{\notparallel}(0)\Vert_{L^2}^2+\left( B_0^2+B_0^{3/2}\right)\left(\frac{\nu}{\lambda_{\nu}}\right)^{1/2}
\Vert c_{\notparallel}(0)\Vert_{L^2}^2\right] \nonumber \\
& \leq C_1\left(\Vert c_{\notparallel}(0)\Vert_{L^2}^2+2\right)\Vert 
c_{\notparallel}(0)\Vert_{L^2}^2, \label{eq:420}
\end{align}
where we have set  $C_1 = 160000 \,(\sqrt{\dfrac{10}{\mu}} +1)^2\, C^\ast$.
Thus,  $y(t):=\|\psi(t)\|_{L^2}^2$ satisfies:
\begin{equation}
y(t)\leq y_0 +\int_0^t p(s)y(s)+q(s)y^{1/2}(s)ds,
\end{equation}
where
\begin{align*}
&p(s)=C^\ast\nu\left(\Vert\Delta c_{\notparallel}\Vert_{L^2}^2\Vert c_{\notparallel}\Vert_{L^2}^2+\Vert\partial_{x_2} c_{\parallel}\Vert_{L^2}^2\Vert\Delta c_{\notparallel}\Vert_{L^2}\Vert c_{\notparallel}\Vert_{L^2}\right);\\
&q(s)=\left(\Vert\partial_{x_2}v\Vert_{L^{\infty}}\Vert\phi\Vert_{L^2}+C^\ast\nu\Vert\partial_{x_2} c_{\parallel}\Vert_{L^2}^3\Vert\Delta c_{\notparallel}\Vert_{L^2}\Vert c_{\notparallel}\Vert_{L^2}\right);\\
&y_0=\|\psi(0)\|_{L^2}^2+C_1\left(\Vert c_{\notparallel}(0)\Vert_{L^2}^2+2\right)\Vert c_{\notparallel}(0)\Vert_{L^2}^2.
\end{align*}
By applying the Gr\"onwall's type inequality in \cite[Theorem 21]{dragomir2003some}, we conclude that
\begin{equation} \label{eqn:yODE}
y(t)\leq \left\{y_0^{1/2}\exp{\left[\frac{1}{2}\int_0^t p(s)ds\right]}+\frac{1}{2}\int_0^t q(s)\exp{\left[\frac{1}{2}\int_s^t p(r)dr\right]}ds\right\}^{2}.
\end{equation}
Next, from \ref{i:bootstrap_a.1}, \ref{i:bootstrap_a.2}, $C_2=20C^\ast$, and Lemma \ref{lem:uniform_parallel} it follows that
\begin{align*}
\int_s^t p(r)dr&\leq \int_0^t p(s)ds \leq C_2\Vert c_{\notparallel}(0)\Vert_{L^2}^2\left(\nu \int_0^t\Vert\Delta c_{\notparallel}(s)\Vert_{L^2}^2ds\right)\\
&\qquad\qquad\qquad+C^\ast\,B_0\sqrt{\frac{10}{\mu}}\left(\frac{\nu}{\lambda_{\nu}}\right)^{1/2}\Vert c_{\notparallel}(0)\Vert_{L^2}^2.
\end{align*}
Choosing $\nu<\widehat{\nu_0}$, then by Lemma \ref{lem:B3} we know that \ref{i:bootstrap_e.3} holds, which yields
\begin{align}
\int_0^t \Vert\phi(s)\Vert_{L^2} ds\leq\frac{10}{\lambda_{\nu}}\Vert\phi(0)\Vert_{L^2},
\end{align}
one then has
\begin{align}
\label{eqn: integ_q}
\int_0^t q(s)ds&\leq \Vert\partial_{x_2}v\Vert_{L^{\infty}}\frac{10\Vert\partial_{x_1}c_{\notparallel}(0)\Vert_{L^2}}{\lambda_{\nu}}+ C^\ast B_0^{3/2}\sqrt{\frac{10}{\mu}}\left(\frac{\nu}{\lambda_{\nu}}\right)^{1/2}\Vert c_{\notparallel}(0)\Vert_{L^2}^2.
\end{align}
We now choose $\nu_{07}$  small enough such that 
\begin{align}
\label{eqn:smallness02}
 C^\ast\sqrt{\frac{10}{\mu}}\,\left(B_0+B_0^{3/2}\right)\left(\frac{\nu}{\lambda_{\nu}}\right)^{1/2}\Vert c_{\notparallel}(0)\Vert_{L^2}^2<1   
\end{align}
for any $\nu<\nu_{07}$. Letting
\begin{align}   
\label{eqn:nu0}
\nu_0=\min\{\nu_{01},...,\nu_{07}\},
\end{align}
we can ensure that \ref{i:bootstrap_e.1}--\ref{i:bootstrap_e.3}, inequalities \eqref{eqn:smallness01} and \eqref{eqn:smallness02} all  hold for $\nu<\nu_0$. Thus, if $c_0$ satisfies \eqref{eqn:smallinitial2}, we also have
\begin{align*}
\int_s^t p(r)dr&\leq \int_0^t p(r)dr \leq C_2\frac{5}{\mu}\Vert c_{\notparallel}(0)\Vert_{L^2}^4+1; \\
\int_0^t q(s)ds&
\leq 10 \Vert\partial_{x_2}v\Vert_{L^{\infty}}+1.
\end{align*}
Now the desired bound follows \eqref{eqn:yODE}:
\begin{align*}
&y(t) \leq \left(y_0^{1/2}+\frac{1}{2}\int_0^t q(s)ds\right)^2\exp{\left[\int_0^t p(s)ds\right]}\\
&\leq 2\left(y_0+\frac{1}{4}\left(\int_0^t q(s)ds\right)^2\right)\exp{\left[\int_0^t p(s)ds\right]}\\
&\leq 2\left(\|\psi(0)\|_{L^2}^2+C_1\left(\Vert c_{\notparallel}(0)\Vert_{L^2}^2+2\right)\Vert c_{\notparallel}(0)\Vert_{L^2}^2+\frac{\left(10\Vert\partial_{x_2}v\Vert_{L^{\infty}}+1\right)^2}{4} \right)\\
&\quad\cdot\exp{\left[C_2\frac{5}{\mu}\Vert c_{\notparallel}(0)\Vert_{L^2}^4+1\right]},
\end{align*}
which completes the proof.
\end{proof}

\begin{rem}
    The proof of the lemma above  shows that 
     the smallness assumption in Theorem \ref{thm:main} is necessary from a technical point of view. As a matter of fact, any bootstrap argument involving only $L^2$ estimates cannot control the cubic nonlinearity. Estimates containing first-order derivatives are needed. Differentiating \eqref{eqn: sheareq} in $x_2$ gives rise a stretching term that can only be handled via the exponential decay of $\|\partial_{x_1}c_{\notparallel}\|_{L^2}$. 
     However, this decay is not sufficient to conclude the argument of the proof without the smallness assumption on $\partial_{x_1}(c_0)_{\notparallel}$ (see \eqref{eqn: integ_q}). Similar smallness conditions on the initial data are used to control
     related  nonlinear PDEs such as  the parabolic-parabolic Keller–Segel system \cite{he2018suppression,zeng2021suppression}.
\end{rem}

\bibliography{CHEShear}
\bibliographystyle{abbrv}

\end{document}